\documentclass[11pt]{article}
\date{}

\usepackage[title]{appendix}
\usepackage{xcolor}
\usepackage[margin=1in]{geometry}                
\usepackage{graphicx}
\usepackage{subcaption}
\usepackage{pdflscape}
\usepackage{amssymb}
\usepackage[normalem]{ulem}
\usepackage{hyperref}
\usepackage{enumitem}
\usepackage{mathrsfs}
\usepackage{epstopdf}
\usepackage{rotating}
\usepackage{longtable} 
\usepackage{adjustbox}
\usepackage{float}

\usepackage{color}
\usepackage{bbm, dsfont}
\usepackage{pst-node}
\usepackage{tikz-cd}
\usepackage{amsfonts} 
\usepackage{geometry}
\usepackage{amsthm}
\usepackage{amsmath}
\usepackage{titlesec}
\usepackage{amssymb}
\usepackage{enumitem}
\usepackage{float}
\usepackage [english]{babel}
\usepackage [autostyle, english = american]{csquotes}
\usepackage{algorithm}
\usepackage[noend]{algpseudocode} 
\makeatletter
\def\BState{\State\hskip-\ALG@thistlm}
\makeatother
\usepackage{hyperref}
\usepackage[normalem]{ulem}
\usepackage{mathrsfs}
\usepackage[italicdiff]{physics}

\newlist{casess}{enumerate}{1}
\setlist[casess]{label=     \textbf{Case} \arabic*:}
\usepackage{mathtools}

\makeatletter
\newcommand*{\rom}[1]{\expandafter\@slowromancap\romannumeral #1@}
\makeatother

\usepackage{etoolbox}

\makeatletter
\patchcmd{\ttlh@hang}{\parindent\z@}{\parindent\z@\leavevmode}{}{}
\patchcmd{\ttlh@hang}{\noindent}{}{}{}
\makeatother

\usepackage{listings}
\usepackage{color} 
\definecolor{mygreen}{RGB}{28,172,0} 
\definecolor{mylilas}{RGB}{170,55,241}

\newlist{Assumptions}{enumerate}{1}
\setlist[Assumptions]{label=     \textbf{Assumption} \arabic*:}

\makeatletter

\newsavebox{\@brx}
\newcommand{\llangle}[1][]{\savebox{\@brx}{\(\m@th{#1\langle}\)}%
  \mathopen{\copy\@brx\kern-0.5\wd\@brx\usebox{\@brx}}}
\newcommand{\rrangle}[1][]{\savebox{\@brx}{\(\m@th{#1\rangle}\)}%
  \mathclose{\copy\@brx\kern-0.5\wd\@brx\usebox{\@brx}}}
\makeatother

\usepackage{lipsum} 
\usepackage{titlesec}
\titleformat{\subsection}[runin]
       {\normalfont\bfseries}
       {\thesubsection}
       {0.5em}
       {}
       [.]

 \newtheorem{thm}{Theorem}[section]
 \newtheorem{cor}[thm]{Corollary}
 
 \newtheorem{lem}[thm]{Lemma}
 \newtheorem{prop}[thm]{Proposition}
 \theoremstyle{definition}
 \newtheorem{defn}[thm]{Definition}
 \theoremstyle{remark}
 \newtheorem{rem}[thm]{Remark}
 \newtheorem{ex}[thm]{Example}
 \numberwithin{equation}{section}

\numberwithin{equation}{section}

\DeclareMathOperator*{\esssup}{ess\,sup}

\newcommand{\rG}{\mathcal{G}}
\newcommand{\B}{\mathcal{B}} 
\newcommand{\C}{\mathbb{C}}

\newcommand{\Q}{\mathbb{Q}}

\def\N{\mathbb{N}}

\def\Z{\mathbb{Z}}
\newcommand{\E}{\mathbb{E}}

\def\H{\mathcal H}
\def\K{\mathcal K}
\def\R{\mathbb{R}}
\def\Z{\mathbb Z}

\def\S{\mathcal S}
\def\e{{\sf e}}

\newcommand{\ol}[1]{\overline{#1}}
\newcommand{\Aut}{{\rm Aut}}
\newcommand{\spn}{{\rm span}}

\newcommand{\U}{\mathcal{U}}
\newcommand{\vN}[1]{\{#1\}''}

\def\sub{\subseteq}
\DeclareMathOperator{\Span}{Span}
\DeclarePairedDelimiterX{\inp}[2]{\langle}{\rangle}{#1, #2}

\newcommand{\eps}{\epsilon}
\newcommand{\subwk}{\subset_{\rm{weak}}}
\makeatletter
\newcommand*\bigcdot{\mathpalette\bigcdot@{.5}}
\newcommand*\bigcdot@[2]{\mathbin{\vcenter{\hbox{\scalebox{#2}{$\m@th#1\bullet$}}}}}
\makeatother

\def\r{{\rm r}}
\def\d{{\rm d}}

\def\si{\sigma}

\def\trace{{\sf tr}}
\def\CC{\mathbb C}

\def\<{\langle}
\def\>{\rangle}
\providecommand{\norm}[1]{\lVert#1\rVert}
\newcommand{\lr}[1]{\left(#1\right)}

\newcommand*\cls[1]{\overline{#1}}

\numberwithin{equation}{section}

\usepackage[backend=biber,maxnames=10]{biblatex}
\begin{document}

\title{Discrete measured groupoid von Neumann algebras via the Gaussian deformation}

\author{Felipe Flores \& James Harbour
\footnote{
\textbf{2020 Mathematics Subject Classification:} Primary 46L10, Secondary 37A20, 47L65.
\newline
\textbf{Key Words:} Measured groupoid, Gaussian action, deformation/rigidity, unbounded cocycle, primeness, fullness, maximal rigid subalgebra.}
}

\maketitle

\begin{abstract}\setlength{\parindent}{0pt}\setlength{\parskip}{1ex}\noindent

Given a discrete measured groupoid $\mathcal{G}$, we study properties of the corresponding von Neumann algebra $L(\mathcal{G})$ using the techniques of Popa's deformation/rigidity theory. More specifically, we define and study the Gaussian deformation associated with any $1$-cocycle of $\rG$ and use it to prove primeness and fullness under appropriate assumptions. We also characterize the maximal rigid subalgebras of $L(\mathcal{G})$ and produce unique prime factorization results for algebras of the form $L(\mathcal{G}_1\times\ldots\times\mathcal{G}_n)$. 
\end{abstract}

\tableofcontents

\section{Introduction}

Since its birth, Popa's deformation/rigidity theory has been extremely useful for the endeavor of obtaining structural properties in von Neumann algebras, especially for group von Neumann algebras. These properties include (but are not limited to) $W^*$-superrigidity \cite{IoPoVa10,CIOS23}, primeness \cite{Pe06,Po08,Ho16}, product rigidity, non-existence/existence and uniqueness of Cartan subalgebras \cite{Va13} and many more. The idea behind this article is to extend this approach -by means of the Gaussian deformation- to the far wider context of groupoid von Neumann algebras, in the hopes of extending older results and possibly finding new phenomena.

One of the structural properties that one often studies using deformation/rigidity is that of primeness. A type $\mathrm{II}$ von Neumann algebra is said to be prime if one cannot write it as the tensor product of two type $\mathrm{II}$ von Neumann algebras. The first primeness result was obtained by Popa \cite{Po83}, where he showed that the group von Neumann algebra of an uncountably-generated free group is prime. Later, in \cite{Ge98} using Voiculescu's work on free entropy dimension, Ge proves that all group factors coming from finitely generated free groups are prime. Then Ozawa, using $C^*$-algebra techniques, greatly generalizes this to show that all i.c.c. hyperbolic groups give rise to prime factors \cite{Oz04}. Using deformation/rigidity, Popa showed that all $\mathrm{II}_1$ factors arising from Bernoulli actions of nonamenable groups are prime \cite{Po08}. Peterson used his derivation approach to deformation/rigidity \cite{Pe06} to prove that any $\mathrm{II}_1$ factor coming from a countable group with positive first $\ell^2$-Betti number is also prime. More recent results on primeness can be found in the work of Hoff \cite{Ho16} on equivalence relations, or in the works of Marrakchi \cite{Ma17} and Patchell \cite{Pa23} on generalized wreath products. 

After primeness, the natural property to study seems to be that of decomposition into prime factors. The idea is to find a decomposition $M=P_1\cls\otimes P_2\cls\otimes \ldots\cls\otimes P_n $, which is unique in the sense that if we have any other decomposition $M={Q_1}\cls\otimes {Q_2}\cls\otimes \ldots\cls\otimes {Q_m} $, then $n=m$ and $P_i$ has to be unitarily conjugated to $Q_i$, modulo permutation of indices and amplifications. A first result like this was obtained by Ozawa and Popa in \cite{OzPo04}, where a combination of Ozawa's solidity results \cite{Oz04} and Popa's intertwining techniques \cite{Po03} is used to show that any $\mathrm{II}_1$ factor arising from a tensor product of hyperbolic group factors has such a unique tensor product decomposition. After this landmark discovery, several other unique prime factorization results have appeared, for example, in the works of Sizemore and Winchester \cite{SiWi13} and Houdayer and Isono \cite{HoIs17}. It has also led to the study of product rigidity (see \cite{Dr21} and the references therein).

Besides our results on primeness and prime factorizations, we also prove fullness of a fairly large class of factors. Fullness of von Neumann algebras was introduced by Connes in \cite{Co74}. We recall that a von Neumann algebra $M$ is called \emph{full} if the inner automorphism group ${\rm Inn} (M)$ is closed in the automorphism group $\Aut (M)$ with respect to the $u$-topology. For any factor $M$ with separable predual, $M$ is full if and only if every norm-bounded asymptotically central sequence of $M$ is asymptotically trivial. Effros showed that the group von Neumann algebra of a non-inner amenable countably infinite group is a full factor \cite{Ef75}. A little later, Choda showed that the crossed products $L^{\infty}(X) \rtimes G$, associated with the same class of groups, are full factors as soon as the action $G\curvearrowright X$ is strongly ergodic, essentially free and probability measure preserving \cite{Ch82}. Some other notable results about fullness can be found in the work of Houdayer and Isono on non-singular actions of bi-exact groups \cite{HoIs16}, which Ozawa later generalized to a bigger class of actions \cite{Oz16}, Hoff's paper on equivalence relation algebras \cite{Ho16} or the paper of Marrakchi on crossed products with factors \cite{Ma20}.

Our main tool in this article is the Gaussian deformation, which we defined for discrete measured groupoids, directly extending previous approaches of Sinclair \cite{Si11}, Peterson and Sinclair \cite{PeSi12} and Hoff \cite{Ho16}. The Gaussian deformation has also appeared in work by Vaes \cite{Va13}, de Santiago, Hayes, Hoff, Sinclair \cite{dSHH21} and de Santiago, DeBonis, and Khan \cite{dSDeKh23} in slightly different fashions. Hoff's work \cite{Ho16} is of particular relevance to us, as the study of equivalence relations naturally precedes the study of general groupoids. The study of general groupoids, however, is a little more involved, and many tools become unavailable. Particular but notable examples include the structural results obtained by Feldman and Moore \cite{FM1:77} and by Connes, Feldman, and Weiss \cite{CoFeWe81}. The reader will notice that the unavailability of these results or appropriate analogs was an obstacle that needed to be circumvented (see, for example, the discussion preceding Lemma \ref{gpbun} or Theorem \ref{full} and its proof).

In any case, we applied this deformation to obtain our first two main theorems, which are the following. 

\begin{thm}\label{mainth}
    Let $\rG$ be a nonamenable ergodic discrete measured groupoid that admits a strongly unbounded 1-cocycle into a mixing orthogonal representation weakly contained in the regular representation. Then $L(\rG)$ is prime. Furthermore, if $\rG$ is strongly ergodic and $L(\rG)$ is a factor, then it is also a full factor.
\end{thm}

\begin{thm}[Corollary \ref{mainth1}]\label{mainth0}
    For $i \in \{1, 2, \ldots, k\}$, let $\rG_i$ be a nonamenable strongly ergodic discrete measured groupoid that admits a strongly unbounded $1$-cocycle into a mixing orthogonal representation weakly contained in the regular representation. Further assume that $L(\rG_i)$ is a factor. Then $M=L(\rG_1 {\times} \rG_2 {\times} \dots {\times} \rG_k) $ satisfies the following.
\begin{enumerate}
    \item[(i)] If $M = N \cls{\otimes} Q$ for tracial factors $N, Q$, there must be a partition $I_N \cup I_Q = \{1, \dots, k\}$ and $t > 0$ such that $N^t = \bigotimes_{i \in I_N} L(\rG_i)$ and $Q^{1/t} = \bigotimes_{i \in I_Q} L(\rG_i)$ modulo unitary conjugacy in $M$.
    \item[(ii)] If $M = P_1 \cls{\otimes} P_2 \cls{\otimes} \cdots \cls{\otimes} P_m$ for $\rm{II}_1$ factors $P_1, \dots, P_m$ and $m \geq k$, then $m = k$, each $P_i$ is prime, and there are $t_1, \dots, t_k > 0$ with $t_1t_2\cdots t_k = 1$ such that after reordering indices and conjugating by a unitary in $M$ we have $L(\rG_i) = P_i^{t_i}$ for all $i$. 
    \item[(iii)] In (ii), the assumption $m \geq k$ can be omitted if each $P_i$ is assumed to be prime.
\end{enumerate}
\end{thm}

As the reader can tell, our theorems are proven for the class of groupoids that admit ``strongly unbounded'' $1$-cocycles, whose definition is a technical strengthening of the (regular) definition of unbounded $1$-cocycles. Let us mention that these two notions coincide in many situations. In the case of transformation groupoids, every unbounded $1$-cocycle induced by the group is strongly unbounded. Moreover, every unbounded $1$-cocycle is strongly unbounded if the groupoid is a group bundle, an equivalence relation or, more generally, a semidirect product groupoid. A discrete measured groupoid is automatically a semidirect product groupoid when the equivalence relation associated with the groupoid is treeable (which includes the case of hyperfinite equivalence relations). This property makes the class of semidirect product groupoids very desirable and it is one of the reasons why they have been gaining some popularity in recent times (see \cite{PoShVa20,BCDKK24}). Let us also mention that the study of treeability is of central importance in measured group theory (see \cite{Ga10}) and percolation theory (see \cite[Section 4.2]{ChIo10}).

Our next result makes use of the framework of maximal rigid subalgebras developed in \cite{dSHH21} to produce a von Neumann algebra-based generalization of a famous result of Peterson and Thom \cite[Theorem 5.6, Corollary 5.8]{PeTh11}. This result was obtained for group algebras by de Santiago, Hayes, Hoff, and Sinclair \cite{dSHH21}.

\begin{thm}[Theorem \ref{bens}] \label{bensmain}
    Let $ \rG $ be an ergodic discrete measured groupoid and $ \pi $ a weak mixing, orthogonal representation of $ \rG $ on a real Hilbert bundle $ X\ast \H $. Suppose that $\rG$ admits an unbounded 1-cocycle into $\pi$ and that $L(\mathcal{G})$ is diffuse. Then, for any pair $P,Q\subset L(\rG)$ of subalgebras with property (T) such that $P\cap Q$ is diffuse, we have that $P\vee Q\not=L(\rG)$.
\end{thm}

Restricting ourselves to the setting of transformation groupoids yields results about crossed products (also called `group measure space von Neumann algebras'), which are interesting by themselves. In such a case, we require that the acting group admits an unbounded $1$-cocycle into a mixing orthogonal representation weakly contained in the regular representation, as this induces a strongly unbounded $1$-cocycle on the transformation groupoid. The main examples of such groups are those groups with positive first $\ell^2$-Betti number. Peterson and Thom noted that the groups with positive first $\ell^2$-Betti number are exactly those that admit an unbounded $1$-cocycle into the (real) left regular representation \cite{PeTh11}. Examples of such groups contain finitely generated free groups, groups which are measure equivalent to them, some one-relator groups (see \cite{DiLi07}) and surface groups. For these groups and their crossed products, we obtained the following corollaries, which are a direct application of the previously stated theorems.

\begin{cor}
   Let $G$ be a countable group that admits an unbounded $1$-cocycle into a mixing orthogonal representation weakly contained in the regular representation, and let $G\curvearrowright (X,\mu)$ be an ergodic, nonamenable probability measure preserving action. Then $L^\infty(X,\mu)\rtimes G$ is prime. Furthermore, if the action is strongly ergodic and $L^\infty(X,\mu)\rtimes G$ is a factor (immediate if $G$ is i.c.c. or if the action is free), then $L^\infty(X,\mu)\rtimes G$ is also a full factor. 
\end{cor}

\begin{cor}
    Let $G$ be a countable group that admits an unbounded $1$-cocycle into a weak mixing orthogonal representation weakly contained in the regular representation, and let $G\curvearrowright (X,\mu)$ be an ergodic probability measure preserving action. Suppose that $L^\infty(X,\mu)\rtimes G$ is diffuse. Then, for any pair $P,Q\subset L^\infty(X,\mu)\rtimes G$ of subalgebras with property (T) such that $P\cap Q$ is diffuse, we have that $P\vee Q\not= L^\infty(X,\mu)\rtimes G$.
\end{cor}

It is worth noting that a crossed product $L^\infty(X,\mu)\rtimes G$ is diffuse in most cases of interest. For example, and because of the ergodicity assumption, this is the case if the action is ergodic, $G$ is infinite and $\mu$ is not supported in a finite set.

As a more concrete example, we found the following primeness result for generalized Bernoulli shifts. For the groups that fit our setting, it generalizes the results of Patchell \cite[Theorem 1.3]{Pa23} and of Marrakchi \cite{Ma17}.

\begin{cor}[Corollary \ref{genbern}]\label{maincor3}
    Let $G$ be a countable group which admits an unbounded $1$-cocycle into a mixing orthogonal representation weakly contained in the regular representation, and let $G\curvearrowright (Y^I,\nu^{\otimes I})$ be a generalized Bernoulli shift. Further suppose that $G\curvearrowright I$ is nonamenable. Then $L^\infty(Y,\mu)^I\rtimes G$ is a prime factor.
\end{cor}

The paper is organized as follows. Section \ref{preliminaries section} contains preliminaries that include the general definitions involving discrete measured groupoids and their von Neumann algebras, unitary/orthogonal representations, $1$-cohomology and the basics on malleable deformations. Section \ref{gaussianc} is devoted to constructing and describing the Gaussian extension and the associated Gaussian deformation that, as we mentioned, will be the central tool for the development of our results. Section \ref{bimoddddd} describes how groupoid representations induce bimodules over the groupoid von Neumann algebra, and establishes a dictionary between some of the properties of these objects. Section \ref{primeness} is devoted to obtaining our primeness result. Section \ref{fullness} is devoted to obtaining conditions that guarantee fullness of the groupoid algebra, completing Theorem \ref{mainth}; this is especially important in obtaining the result about unique prime factorization and makes Section \ref{UPF} very short. Indeed, Section \ref{UPF} only contains a theorem of Hoff \cite{Ho16} and its application which, due to the previous work, immediately yields Theorem \ref{mainth0}. Finally, in Section \ref{maxrigsec}, we characterize the maximal rigid subalgebras of the deformation in terms of the groupoid. This is then used to prove Theorem \ref{bensmain}.

\section{Preliminaries}
\label{preliminaries section}

\subsection{Discrete measured groupoids}\label{meas-group}

We will work with groupoids $\rG$ over a unit space $X\equiv\rG^{(0)}$, identified as small categories in which all the morphisms (arrows) are invertible. 
The domain and range maps are denoted by ${\rm d,r}:\rG\to \rG^{(0)}$ and the family of composable pairs by $\,\rG^{(2)}\!\subset\rG\times\rG$\,.  For $x\in\rG\,,\,A, B\subset\rG$ we use the notations
\begin{equation}\label{botations}
AB\coloneqq\{ab\,:\,a\in A,b\in B,\d(a)=\r(b)\},
\end{equation}
\begin{equation*}
Ax\coloneqq A\{x\}\quad \textup{and}\quad xA\coloneqq\{x\}A
\end{equation*}
These sets could be void in non-trivial situations. A subset of the unit space $Y \subset \rG^{(0)}$ is called {\it invariant} if $Y=\r(\rG Y)$.

Suppose that $\rG$ is a groupoid equipped with the structure of a standard Borel space 
such that the composition and the inverse map are Borel and 
$\d^{-1}(\{x\})$ is countable for all $x\in\rG^{(0)}$. Then the domain and range maps are measurable, $\rG^{(0)}\subset\rG$ is 
a Borel subset, and $\d^{-1}(\{x\})$ is countable.

Now let $\mu$ be a probability measure on 
the set of units $\rG^{(0)}$. Then, 
for any measurable subset $A\subset\rG$, the function 
$\rG^{(0)}\ni x\mapsto \#\bigl (\d^{-1}(x)\cap A\bigr )$ 
is measurable, and the measure $\mu_\d$ on $\rG$ defined by 
$$
\mu_\d(A)=\int_{\rG^{(0)}} \#\bigl (\d^{-1}(x)\cap A\bigr )d\mu (x)=\int_{\rG^{(0)}} \#\bigl (Ax)d\mu (x)
$$
is $\sigma$-finite. The measure $\mu_\r$ is defined in an analogous manner, replacing $\d$ by $\r$. 

\begin{prop}\label{invr} Let $i: x\to x^{-1}$ be the inversion map in $\rG$. The following conditions on $\mu$ are equivalent. 
\begin{enumerate}
\item $\mu_\d=\mu_\r$, 
\item $i_*\mu_\d=\mu_\d$, 
\item for every Borel subset $E\subset\rG$ such that 
$\d\vert_{ E}$ and $\r\vert_{E}$ are injective we have 
$\mu(\d(E))=\mu(\r(E))$. 
\end{enumerate}
\end{prop}
Such a probability measure is called {\it invariant} and we denote $\mu_\rG=\mu_\r=\mu_\d$.

\begin{defn}\label{measuredgroupoid}
A discrete, measurable groupoid $\rG$ together with an invariant
probability measure 
on $\rG^{(0)}$ is 
called a {\it discrete measured groupoid}. 
\end{defn}

For $A\subset\rG^{(0)}$ one uses the standard notations 
\begin{equation*}\label{faneaka}
\rG_A\!\coloneqq\d^{-1}(A)\,,\quad\rG^A\!\coloneqq{\rm r}^{-1}(A)\,,\quad\rG_A^A\coloneqq\rG_A\cap\rG^A.
\end{equation*} If $A$ is Borel, then we equip $\rG_A^A$ with the normalized measure $\frac{1}{\mu(A)}\mu\vert_{A}$, so it becomes 
a discrete measured groupoid, called the {\it restriction} of 
$\rG$ to $A$ and is denoted by $\rG\vert_A$. Note that every unit $x\in X$ gives rise to a subgroup $\rG_x^x$, called an \emph{isotropy subgroup}. The union of all isotropy subgroups is naturally a group bundle, we denote it by ${\rm Iso}(\rG)$ and call it the \emph{isotropy subgroupoid} of $\rG$.

As usual, the set of complex-valued, measurable, essentially bounded functions (modulo almost everywhere null functions) on $\rG$ with respect to $\mu_\rG$ is denoted by $L^\infty (\rG,\mu_\rG)$. For a function $\phi:\rG\rightarrow\CC$ and $x\in\rG^{(0)}$ we put \begin{equation*}\begin{aligned}
D(\phi)(x)&=\#\left\{g\in\rG:\phi(g)\ne
  0,\d(g)=x\right\},\\ 
R(\phi)(x)&=\#\left\{g\in\rG:\phi(g)\ne
  0,\r(g)=x\right\}.
\end{aligned}
\end{equation*}
The {\it groupoid ring} $\CC\rG$ of $\rG$ is defined as 
\[
\CC\rG=\left\{\phi\in L^\infty(\rG,\mu_\rG): \text{$D(\phi)$ and $R(\phi)$ are essentially bounded on $\rG^{(0)}$}\right\}.
\]
$\CC\rG$ is a $^*$-algebra containing $L^\infty(\rG^{(0)})\equiv L^\infty(\rG^{(0)},\mu)$ as a $^*$-subalgebra. Multiplication is given by the convolution product 
\[
\phi*\eta(x)=\sum_{yz=x}\phi(y)\eta(z)
\]
and the involution is defined by 
\[
\phi^*(x)=\overline{\phi(x^{-1})}.
\]
The groupoid ring $\CC\rG$ of discrete measured groupoid $\rG$ is a weakly dense $^*$-subalgebra in the {\it von Neumann algebra} $L(\rG)$ of $\rG$. In fact, $L(\rG)$ is defined as the WOT-closure of $\CC\rG$, when the latter is $^*$-represented as convolution operators 
\[
L_\psi(\xi)(x)=\psi*\xi(x)= \sum_{yz=x}\psi(y)\xi(z),\quad\textup{   for }\psi\in \CC\rG, \xi\in L^2(\rG,\mu_{\rG}).
\]
The von Neumann algebra $L(\rG)$ has a canonical tracial state $\trace_{L(\rG)}$, which is induced by the invariant measure $\mu$. For $\phi\in\CC\rG\subset L(\rG)$, it is given by the formula 
\[
\trace_{L(\rG)}(\phi)=\int_{\rG^{(0)}}\phi(x)d\mu(x).
\]

\begin{defn}
  A (Borel) \textit{bisection} of $ \rG $ is a Borel subset $ \sigma\sub \rG $ such that the sets $ \sigma x $ and $ x\sigma $ have at most $1$ element for every $ x\in G^{0} $. Borel bisections form an inverse semigroup with respect to the operation on sets described in \eqref{botations}.

  The \textit{full pseudogroup} $ [[\rG]] $ of $ \rG $ is the inverse semigroup consisting of Borel bisections modulo the relation of being equal almost everywhere. The \textit{full group} $ [\rG] $ is the subset of $ [[\rG]] $ consisting of the Borel bisections $ \sigma $ such that $ \sigma\sigma^{-1} = \sigma^{-1}\sigma = \rG^{(0)} $. When $\sigma\in [\rG]$, $ \sigma x $ and $ x\sigma $ have exactly one element, so we identify them with said element.

  For any pseudogroup element $\sigma\in [[\rG]]$, we define the partial isometry $u_\sigma:= L_{1_\sigma}\in L(\rG)$ as left convolution by the indicator function of $\sigma\sub \rG$. When $\sigma$ is a full group element, $u_\sigma$ is a unitary in $L(\rG)$. 
\end{defn}


\begin{defn}[\cite{ADRe00}]
    The groupoid $\rG$ is called \emph{amenable} if there exists a norm one projection $M:L^\infty(\rG, \mu_\rG)\to L^\infty(X,\mu)$ with $M(\phi*\psi)= \phi*M(\psi)$ for all $\psi \in L^\infty(\rG,\mu_\rG)$, and all Borel functions $\phi$ on $\rG$ such that $\sup_{x\in X}\sum_{g\in\rG^x}|\phi(g)|<\infty$.
\end{defn}

It is known that a discrete measured groupoid $\rG$ is amenable if and only if its von Neumann algebra $L(\rG)$ is amenable \cite[Corollary 6.2.12]{ADRe00}. Moreover, we say that $\rG$ has an amenable direct summand if there is a measurable subset $Y \subset X$ such that $\mu(Y) > 0$, $\rG\vert_Y$ is amenable, and $\rG = \rG\vert_Y \cup \rG\vert_{Y^c}$. In this case, $L(\rG)$ has an amenable direct summand as well.

We now turn our attention to some of the main examples of discrete measured groupoids and their von Neumann algebras. Subsection \ref{semidirekt} will introduce another important class of examples. For more examples, we recommend looking at \cite{So24} or at the last section of \cite{BCDK24}.

\begin{ex}\label{transformation}
Suppose $G$ is a countable group acting (in a measure-preserving way) on the standard probability space $(X,\mu)$, denoted by $g\cdot x\equiv g\cdot_\theta x$, for $g\in G, x\in X$. The {\it transformation groupoid} $\rG\coloneqq X\rtimes_\theta G$ has $G\times X$ as the underlying set. The multiplication is 
$$
(g,x)(h,g^{-1}\cdot x)\coloneqq(gh,x)$$ and inversion reads $$(g,x)^{-1}\coloneqq\big(g^{-1}\!,g^{-1}\cdot x\big).
$$ 
So $\rG^{(0)}=\{1_G\}\times X$ gets identified with $X$, so $\r(g,x)=x$ and $\d(g,x)=g^{-1}\cdot x$. In this case, $L(\rG)\cong L^\infty(X,\mu)\rtimes_\theta G$ with its usual trace. In particular, if $X=\{x_0\}$ is a singleton, this construction gives the group algebra $L(G)$ again with its usual trace. 
\end{ex}

\begin{defn}\label{zimmer}
    We will say that the action $G\curvearrowright (X,\mu)$ is \emph{amenable}, or even \emph{Zimmer amenable} if the associated transformation groupoid $X\rtimes G$ is amenable.
\end{defn}

\begin{ex}\label{equivrel}
Let $(X,\mu)$ be a standard probability space and $\mathcal R\subset X\times X$ be an equivalence relation which is a measurable subset. $\mathcal R$ becomes a discrete measured groupoid with the operations 
$$
\d(x,y)=(y,y)\,,\quad \r(x,y)=(x,x)\,,\quad (x,y)(y,z)=(x,z)\,,\quad(x,y)^{-1}\!=(y,x)\,.
$$
The unit space is $\mathcal R^{(0)}={\sf Diag}(X)$ and we identify it with $X$, via the map $(x,x)\mapsto x$. We say that $\mathcal R$ is measure preserving if the resulting groupoid is a discrete measured groupoid. The algebra $L(\mathcal R)$ introduced here coincides with the usual equivalence relation von Neumann algebra (typically introduced using the full pseudogroup). For example, see \cite[Section 2.2]{Ho16}. 
\end{ex} 

\begin{rem}\label{nonfree}
A particular way to obtain equivalence relations is via group actions. Namely, if $G$ is a countable group acting in a measure-preserving way on the probability space $(X,\mu)$, we can form the equivalence relation $\mathcal R_{G,X}$ given by the orbits of the action, more explicitly
\[
\mathcal R_{G,X}=\{(x,y)\in X\times X : y=g\cdot x, \textup{ for some }g\in G\},
\]
with the inherited Borel structure of $X\times X$. It is a famous result of Feldman and Moore that every discrete measured equivalence relation can be realized as $\mathcal R_{G,X}$ for some probability measure-preserving group action \cite[Theorem 1]{FM1:77}.
    
It is worth noting that, when the action is essentially free (meaning that for almost all $x\in X$, the stabilizer of $x$ is trivial), the groupoids $\mathcal R_{G,X}$ and $X\rtimes_\theta G$ are isomorphic up to a measure zero set and so they generate the same von Neumann algebra (cf. \cite{FM2:77}). 
\end{rem}


\begin{defn}
    A discrete measured groupoid $\rG$ is called {\it ergodic} if $\mu(Y)\in\{0,1\}$ for every Borel invariant subset $Y\subset \rG^{(0)}$.
\end{defn}

Let $L^\infty(X,\mu)^\rG$ denote the algebra of $L^\infty$ functions fixed by $\rG$. That is, every $a\in L^\infty(X,\mu)^\rG$ satisfies 
\[
a(\si^{-1}x)=a(x) \text{ for almost all } x\in X \text{ and all }\si\in[\rG].
\] It is clear that $L^\infty(X,\mu)^\rG=\CC1$ if and only if $\rG$ is ergodic. We also note that $L^\infty(X,\mu)^\rG$ lies in the center $\mathcal{Z}(L(\rG))=L(\rG)'\cap L(\rG)$ of $L(\rG)$. Therefore, the ergodicity of $\rG$ is necessary for $L(\rG)$ to be a factor. In a recent preprint, Berendschot, Chakraborty, Donvil and Kim characterize the groupoids that give a factor \cite{BCDK24} by introducing and studying the notion of \emph{i.c.c.} groupoids.

We will deal mostly with ergodic groupoids, so it seems convenient to examine when our examples satisfy this condition.

\begin{rem}
The transformation groupoid introduced in Example \ref{transformation} is ergodic precisely when the action $\theta$ is ergodic. Moreover, note that, for $(g,x)\in \rG$ and $y\in Y$ with $\d(g,x)=y$, we have 
$$
\r\big((g,x)(1_G,y)\big)=\r\big((g,x)\big)=x=g\cdot_\theta y, 
$$
so $\r(\rG Y)=G\cdot_\theta Y$ and $X\rtimes_\theta G$ is ergodic if and only if $\mu(G\cdot_\theta Y)\in\{0,1\}$ for every Borel subset $Y\subset X$. In particular, every group is an ergodic groupoid.
\end{rem}

\begin{rem}
In the case of equivalence relation groupoids $\mathcal R$ (Example \ref{equivrel}), we have $$
\r\big((x,y)(y,y)\big)=\r(x,y)=x, 
$$
so $\r(\rG Y)=\{x\in X: x\sim_{\mathcal R} y,\textup{ for some }y\in Y\}=: [Y]_{\mathcal R}$ and $\mathcal R$ is ergodic if $\mu([Y]_{\mathcal R})\in\{0,1\} $, for every Borel subset $Y\subset X$.

\end{rem}


\subsection{Semidirect product groupoids} \label{semidirekt}

A discrete measured groupoid $(\mathcal{G},\mu)$ comes with an associated discrete Borel equivalence relation $\mathcal{R}_\rG$, defined on $(X,\mu)$ with $X=\mathcal{G}^{(0)}$. It is, in fact, given by 
\[
\mathcal{R}_\rG\coloneqq\{(\r(g),\d(g)) : g\in\mathcal{G}\}.
\]
It also comes with a canonical quotient map 
\begin{equation*}
    q:\rG\to\mathcal{R}_\rG,\quad\text{given by}\quad q(g)=\big(\r(g),\d(g)\big).
\end{equation*}
The purpose of this section is to introduce a special class of groupoids, called semidirect product groupoids. These could be characterized as the groupoids for which the above-mentioned quotient map admits a right-inverse measurable map $\mathcal R_\rG\to\rG$ which is also a homomorphism. 

In what follows, we understand a group bundle $\Gamma=(G_x)_{x\in X}$ simply as a Borel groupoid where the maps $\d$ and $\r$ coincide. In this case, the group $G_x$ is precisely the isotropy subgroup associated to the unit $x\in X$.

Now consider any discrete Borel equivalence relation $\mathcal{R}$ on a standard probability space $(X,\mu)$, and a group bundle of discrete countable groups $\Gamma\coloneqq(G_{x})_{x\in X}$. An action of $\mathcal{R}$ on $\Gamma$ is given by a family of group isomorphisms $\delta_{(y,x)}: G_{x}\rightarrow G_{y}$ for all $(y,x)\in\mathcal{R}$ such that $\delta_{(z,y)}\circ\delta_{(y,x)}=\delta_{(z,x)}$, $\delta_{(y,x)}^{-1}=\delta_{(x,y)}$, and $\delta{(x,x)}=\text{id}_{G_{x}}$ hold $\mu$-almost everywhere. 

In the above setting, one defines the \emph{semidirect product groupoid} $\rG=\Gamma\rtimes_\delta \mathcal{R}$ as the set 
$$
\{((y,x),g)\in \mathcal R\times\Gamma\mid g\in G_{x}\text{ and }(y,x)\in\mathcal{R}\},
$$ 
equipped with the natural measurable structure and the operations  
\[
((z,y), h)\circ((y,x), g)=((z,x), \delta_{(x,y)}(h)\cdot g)
\]
and
\[
((y,x), g)^{-1}=((x,y), \delta_{(y,x)}(g^{-1}))\,.
\]
It then follows that the domain and range maps can be described by the projections $\d((y,x),g)= x$ and $\r((y,x),g)= y$, after the -natural- identification of the unit space $\rG^{(0)}=\{((x,x), 1_{G_x})\in \mathcal{R}\times\Gamma: x\in X\}$ with $X$.

Recall now that a discrete Borel equivalence relation $\mathcal{R}$ on $(X,\mu)$ is called \emph{treeable} if there exists a Borel graph $T$ on $(X,\mu)$ such that $T\subseteq\mathcal{R}$ and if for $\mu$-almost every $x\in X$ the connected component of $x$ in $T$ is a tree with vertex set $T\cdot x$. Hyperfinite equivalence relations provide the easiest example of treeable equivalence relations (cf. \cite{CoFeWe81}). In \cite[Proposition 6.5]{PoShVa20}, Popa-Shlyakhtenko-Vaes proved the following proposition, of great utility for us.

\begin{prop} Let $\mathcal R$ be a discrete measured equivalence relation on the standard probability space $(X, \mu)$. The following two statements are equivalent. \begin{itemize}
    \item[(i)] $\mathcal R$ is treeable.
    \item[(ii)] Any discrete measured groupoid $\rG$ with $\mathcal R_\rG\cong\mathcal R$ is isomorphic to a semidirect product groupoid. 
\end{itemize}
\end{prop}

\subsection{Unitary representations and \texorpdfstring{$1$}--cohomology}

Given a collection of Hilbert spaces $\{\H_x\}_{x \in X}$, the Hilbert bundle $X \ast \H$ is the set of pairs $X \ast \H = \{(x, \xi_x) : x \in X, \xi_x \in \H_x\}$. 
A section $\xi$ of $X \ast \H$ is a map $x \mapsto \xi_x \in \H_x$. 

A {\it measurable Hilbert bundle} is a Hilbert bundle $X \ast \H$ endowed with a $\sigma$-algebra generated by the maps $\{(x, \xi_x) \mapsto \<\xi_x, \xi^n_x\>\}_{n = 1}^\infty$ for a {\it fundamental sequence of sections} $\{\xi^n\}_{n = 1}^\infty$ satisfying 

$(i)$ $\H_x = \ol{\spn \{\xi^n_x\}_{n = 1}^\infty}$ for each $x \in X$, and 

$(ii)$ the maps $\{x \mapsto \|\xi^n_x\|\}_{n = 1}^\infty$ are measurable. 

It is a useful fact that the $\sigma$-algebra of any measurable Hilbert bundle can be generated by an {\it orthonormal fundamental sequence of sections}, i.e. sections which moreover satisfy 

$(iii)$ $\{\xi^n_x\}_{n = 1}^{\infty}$ is an orthonormal basis of $\H_x$ for $x \in X$ with $\dim \H_x = \infty$, and if $\dim \H_x < \infty$, the sequence $\{\xi^n_x\}_{n = 1}^{\dim \H_x}$ is an orthonormal basis and $\xi^n_x = 0$ for $n > \dim \H_x$. 
\\

A {\it measurable section} of $X \ast \H$ is a section $\xi$ such that $x \mapsto (x,\xi_x) \in X \ast \H$ is a measurable map, or equivalently, such that the maps $\{x \mapsto \<\xi_x, \xi^n_x\>\}_{n = 1}^\infty$ are measurable for the fundamental sequence of sections $\{\xi^n\}_{n = 1}^{\infty}$. We let $S(X \ast \H)$ denote the vector space of measurable sections, identifying $\mu$-a.e. equal sections. It is also useful to reserve some notation for the sections with constant norm: 
$$
S_1(X \ast \H)=\{\xi\in S(X \ast \H) : \norm{\xi_x}_{\H_x}=1 \textup{ a.e.}\}.
$$ The elements in $S_1(X \ast \H)$ are called {\it normalized } sections. As hinted, we will often abuse the notation and confuse the map $x \mapsto (x, \xi_x)$ with $\xi$.  We then consider the {\it direct integral} 
\begin{align*}
\int_X^{\oplus} \H_x d\mu(x) = \{\xi \in S(X \ast \H) : \int_X \|\xi(x)\|^2 d\mu(x) < \infty\}
\end{align*}
which is a Hilbert space with inner product $\<\xi, \eta\> = \int_X \<\xi_x, \eta_x\> d\mu(x)$. If $a \in L^\infty(X)$ and $\xi \in \int_X^{\oplus} \H_x d\mu(x)$ we denote by $a\xi$ or $\xi a$ the element of $\int_X^{\oplus} \H_x d\mu(x)$ given by $[a\xi](x) = [\xi a](x) = a(x)\xi_x$. If $\{\xi^n\}_{n = 1}^{\infty}$ is an orthonormal fundamental sequence of sections, any $\xi \in \int_X^{\oplus} \H_x d\mu(x)$ with $\esssup_{x\in X} \norm{\xi_x}<+\infty$ has an expansion $\xi = \sum_{n = 1}^\infty a_n\xi^n$ where $a_n\in L^\infty(X)$ is given by $a_n(x) = \<\xi_x, \xi^n_x\>_{\H_x} $. 

\begin{defn}\label{unitaryrep}
A {\it unitary (resp. orthogonal) representation} of $\rG$ on a complex (real) measurable Hilbert bundle $X \ast \H$, with $X=\rG^{(0)}$ and a map $\rG\ni g\mapsto \pi(g) \in \U(\H_{\d(g)}, \H_{\r(g)})$ (in the real case, $\U(\H, \K)$ denotes the set of orthogonal maps from $\H$ onto $\K$) such that
$$
\pi(gh)=\pi(g)\pi(h), \quad\textup{ for almost all } (g,h)\in\rG^{(2)}
$$
and such that $\rG\ni g \mapsto \<\pi(g)\xi_{\d(g)}, \eta_{\r(g)}\>$ is a measurable map for all $\xi, \eta \in S(X \ast \H)$. 
\end{defn}

\begin{ex}
Given a measurable Hilbert bundle $X \ast \K$ with orthonormal fundamental sequence $\Xi=\{\xi^n\}_{n=1}^\infty$, one can always define the {\it identity representation} ${\rm id}_\Xi$ by the formula 
$$
{\rm id}_\Xi(g)\xi^n_{\d(g)}=\xi^n_{\r(g)},
$$ for each $g\in\rG$, $n\in\N$.
\end{ex}

\begin{ex} \label{leftreg}
The {\it (left) regular representation} $\lambda_\rG$ of $\rG$ is obtained by taking $\H_x = \ell^2(\rG^x)$ for each $x \in X=\rG^{(0)}$, forming the measurable Hilbert bundle $X \ast \H$ with any fundamental sequence such that $S(X \ast \H)=\{\xi\textup{ is a measurable function}: \xi_x\in \H_{\r(x)}\}$, and considering the identity representation with respect to this fundamental sequence. An example of such a fundamental sequence is achieved by fixing a topology in $\rG$ that generates the Borel structure and taking any sup-norm dense sequence $\{f_n\}_{n=1}^\infty$ in $\mathcal C_{\rm c}(\rG)$ to define $\xi_x^n=f_n\vert_{\rG^x}$.

In this case, one has in a natural way that $$
\int_X^\oplus \H_x\dd \mu(x)\cong L^2(\rG,\mu_\rG) 
$$ and the action of $\rG$ is given by 
$$
\lambda_\rG(g)\xi(h)=\xi(g^{-1}h),\quad\text{ for } (g,h)\in\rG^{(2)} \text{ and }\xi\in \ell^2(\rG^x). 
$$
More suggestively, for $ g\in \rG $ let $ \delta_{g}\in \ell^{2}(\rG^{\r(g)}) $ be the indicator function of $ \{g\}\sub \rG^{\r(g)}$. Then $ \lambda_\rG(g) \delta_{h} = \delta_{gh} $ for all $ (g,h)\in \rG^{(2)}$. This induces a representation $ [[\lambda_\rG]] $ of $ [[\rG]] $ on $ L^{2}(\rG) $. Then $ L(\rG) $ may also be defined as the von Neumann algebra generated by the partial isometries $ v_{\sigma} \coloneqq [[\lambda_\rG]](\sigma)\in \mathbb B(L^{2}(\rG,\mu_\rG)) $.
\end{ex}

\begin{ex}\label{tensors}
    Given two representations $\pi_i$ on $X \ast \H^i$ ($i=1,2$), one may form their tensor product $\pi_1\otimes\pi_2$ by constructing the Hilbert bundle $X \ast \H$, with fibers $\H_x=\H_x^1\otimes \H_x^2$ and a fundamental sequence given by $\xi^{i,j}=\xi^i_1\otimes \xi_2^j$, in terms of the fundamental sequences $\{\xi_i^n\}_{n=0}^\infty$ of $X\ast \H^i$. The formula for $\pi_1\otimes\pi_2$ on simple tensors is $$\pi_1\otimes\pi_2(g)(\xi_1\otimes\xi_2)=\pi_1(g)\xi_1\otimes\pi_2(g)\xi_2.$$
\end{ex}

\begin{defn} \label{equiv}
Given representations $\pi$ on $X \ast \H$ and $\rho$ on $X \ast \K$, we say that $\pi$ and $\rho$ are {\it unitarily equivalent} if there is a family of unitaries $\{U_x \in \U(\H_x, \K_x)\}_{x \in X}$ with 
$$
U_{\r(g)}\pi(g) = \rho(g)U_{\d(g)} \quad \text{for all }g \in \rG,
$$
and such that $x \mapsto U_x\xi_x$ is in $S(X \ast \K)$ for each $\xi \in S(X \ast \H)$. 
\end{defn}

Unitary equivalence allows us to identify two representations as the same. An example of a natural unitary equivalence is presented below.

\begin{lem}[Fell's Absorption Principle for Groupoids]
  Let $ \pi $ be a representation of $ \rG $ on $ X\ast \H $ and $ \lambda_\rG $ the left regular representation of $ \rG $. Then for any orthonormal fundamental sequence of sections $ \Xi=\{\xi^n\}_{n=1}^\infty $ for the bundle $ X\ast \H $, we have that $ \pi\otimes \lambda_\rG $ is unitarily equivalent to $ {\rm id}_{\Xi}\otimes \lambda_\rG $
\end{lem}

\begin{proof}
  Let $ \Xi = \{\xi^{n}\}_{n=1}^{\infty} $ be an orthonormal fundamental sequence of sections for $ X\ast \H $. For $ g,h\in \rG $ and $n, m\in\N $, we compute
  \begin{align*}
    \inp{\pi(g) \xi_{\d(g)}^{n}\otimes \delta_{g}}{\pi(h) \xi_{\d(h)}^{m}\otimes \delta_{h}} &= \inp{\pi(g) \xi_{\d(g)}^{n}}{\pi(h) \xi_{\d(h)}^{m}}\cdot \inp{\delta_{g}}{\delta_{h}}\\
    &= \inp{\pi(g) \xi_{\d(g)}^{n}}{\pi(h) \xi_{\d(h)}^{m}}\cdot \delta_{g=h} \\
    &= \inp{\pi(g) \xi_{\d(g)}^{n}}{\pi(g) \xi_{\d(g)}^{m}}\cdot \delta_{g=h} \\ &=\delta_{g=h}\cdot \delta_{n=m} \\
    &= \inp{\xi_{\r(g)}^{n}\otimes \delta_{g}}{\xi_{\r(h)}^{m}\otimes \delta_{h}}
  \end{align*}

  So, for every $ x\in X $, setting $ U_{x}(\xi_{x}^{n}\otimes \delta_{g}) = \pi(g) \xi_{\d(g)}^{n} \otimes \delta_{g} $ defines a unitary on $ \H_{x}\otimes \ell^{2}(\rG^{x}) $. Now let $(g,h)\in \rG^{(2)}$, $n\in \N$, we see that 
\begin{align*}
    U_{\r(g)}\big({\rm id}_\Xi\otimes \lambda_\rG\big)(g)[\xi_{\r(h)}^n\otimes\delta_h]&=U_{\r(g)}(\xi_{\r(g)}^n\otimes\delta_{gh}) \\
    &=\pi(gh)\xi_{\d(h)}^n\otimes\delta_{gh} \\
    &=\big(\pi\otimes \lambda_\rG\big)(g)[\pi(h)\xi_{\d(h)}^n\otimes\delta_h] \\
    &=\big(\pi\otimes \lambda_\rG\big)(g)U_{\d(g)}[\xi_{\r(h)}^n\otimes\delta_h].
\end{align*} 
The computation above implies that $U_{\r(g)}\big({\rm id}_\Xi\otimes \lambda_\rG\big)(g)=\big(\pi\otimes \lambda_\rG\big)(g)U_{\d(g)}$, for all $g\in\rG$. The fact that $\{U_x\}_{x\in X}$ maps measurable sections into measurable sections is obvious. \end{proof}

\begin{defn}\label{weakcont}
We say that $\pi$ is {\it weakly contained} in $\rho$, denoted $\pi \prec \rho$, if for any $\varepsilon > 0$, $\xi \in S(X \ast \H)$, and $E \subset \rG$ with $\mu_\rG(E) < \infty$, there exists $\{\eta^1, \dots, \eta^m\} \subset S(X \ast \K)$ with
$$
\mu_\rG\Big(\{g\in E : |\< \pi(g)\xi_{\d(g)}, \xi_{\r(g)} \> - \sum_{i = 1}^m \< \rho(g)\eta^i_{\d(g)}, \eta^i_{\r(g)} \>| \geq \varepsilon\}\Big) < \varepsilon. 
$$ 
\end{defn}

\begin{defn}
    Let $\pi$ be a representation of the discrete measured groupoid $\rG$ on a Hilbert bundle $X\ast \H$. We say that a measurable section $\xi\in S(X\ast \H)$ is \emph{invariant} if 
    \[
    \pi(g)\xi_{\d(g)}=\xi_{\r(g)}, \text{ for almost all }g\in\rG.
    \]
    We say that $\pi$ is ergodic if the subspace of invariant sections in $\int_X^\oplus \H_x\dd \mu(x)$ is trivial.
\end{defn}

Now we introduce the definitions of weakly mixing and mixing representations. They were taken from \cite[Definition 3.10]{GaLu17} and \cite[Definition 4.4]{Ki17}, respectively.

\begin{defn}\label{mixingreps}
  Let $\pi$ be a representation of a discrete measured groupoid $\rG$ on a Hilbert bundle $X\ast \H$. Then $\pi$ will be called \begin{enumerate}
      \item[(i)]  \textit{weak mixing} if, for every $\varepsilon >0$ and every $n\in \N$, and sections $ \eta^1,\ldots, \eta^n\in S(X\ast\H) $, there exists $ \si\in [\rG] $ such that
\[
  \int_{X} |\langle\eta^i_x,\pi(x\si)\eta_{\d(x\sigma)}^j\rangle| \dd{\mu(x)} \leq \varepsilon
\]
for every $i,j=1,\ldots,n$.
\item[(ii)] {\it mixing} or $c_0$ if 
for every $\varepsilon, \delta > 0$ and every pair of normalized sections $\xi, \eta \in S(X \ast \H)$, there is $E \subset X$ with $\mu(X \setminus E) < \delta$ such that 
$$
\left|\{g \in \rG^{E}_x : |\<\pi(g)\xi_x, \eta_{\r(g)}\>| > \varepsilon\}\right| < \infty  \quad \text{for $\mu$-a.e. $x \in E$.}
$$
  \end{enumerate} 
\end{defn}

\begin{defn}\label{cohom}
A {\it $1$-cocycle} for a representation $\pi$ on $X \ast \H$ is a measurable map $\rG\ni g \mapsto b(g) \in \H_{\r(g)}\subset X\ast \H$ such that 
\[ 
b(gh) = b(g) + \pi(g)b(h) \quad \text{for all } (g,h)\in \rG^{(2)}.
\] 
The $1$-cocycle $b$ is a {\it $1$-coboundary} if there is a measurable section $\xi$ of $X \ast \H$ such that 
\[
b(g) = \xi_{\r(g)} - \pi(g)\xi_{\d(g)} \quad \textup{for  $\mu_\rG$-a.e. $g\in\rG$}.
\]
A pair of $1$-cocycles $b$ and $b'$ are {\it cohomologous} if $b - b'$ is a $1$-coboundary. The set of $1$-cocycles of $\pi$ is denoted $Z^1(\rG,\pi)$ and the set of $1$-coboundaries by $B^1(\rG,\pi)$. The quotient 
\[
H^1(\rG,\pi)=Z^1(\rG,\pi)/B^1(\rG,\pi)
\]
is called the $1$-cohomology group of the representation $\pi$ and it is typically endowed with the quotient topology after giving $Z^1(\rG,\pi)$ the topology of convergence in the measure $\mu$.
\end{defn}

The following result is due to \textcite[Theorem 3.19, Lemma 3.20]{An05}.

\begin{lem}[Anatharaman-Delaroche]\label{bound}
Let $b$ be a 1-coboundary associated with the representation $\pi$ of $\rG$ in $X\ast \H$, then there exists a Borel subset $E\subset X$ of positive measure, such that for every $x\in E$, $\sup\{\norm{b(g)}: g\in \rG_E^x\}<\infty$. If $\rG$ is ergodic, then this last condition is equivalent to $b$ being a 1-coboundary and, moreover, $E$ can be chosen to have measure $1$.
\end{lem}

However, we are interested in a particular class of $1$-cocycles that we will introduce now. 
\begin{defn}
    Let $b$ be a $1$-cocycle associated with the representation $\pi$ of $\rG$ in $X\ast \H$. Then $b$ is said to be \emph{bounded} if there exists a sequence of measurable subsets $\{E_n\}_{n = 1}^\infty$ of $X$ with $\mu(\bigcup_{n = 1}^\infty E_n) = 1$ and $\sup\{\|b(g)\|: g \in \rG\vert_{E_n}\} < \infty$ for each $n \geq 1$. On the contrary, $b$ is called \emph{unbounded} if it is not bounded. Moreover, $b$ is called \emph{strongly unbounded} if there exists a $ \delta >0 $ such that for all $ R>0 $ there is a $ \sigma\in [\rG] $ such that 
    $$
        \mu(\{x\in X : \norm{b(x\sigma)} \geq R\}) > \delta.
    $$
\end{defn}

In the equivalence relation setting, Hoff characterized coboundaries as unbounded $1$-cocycles \cite[Lemma 2.1]{Ho16}. His proof also works here.
\begin{lem}\label{unboundedness}
A $1$-cocycle $b$ for a representation $\pi$ of $\rG$ on $X \ast \H$ is a coboundary if and only if it is bounded.
\end{lem}
\begin{proof}
If $b$ is a coboundary, then there is a $\xi \in S(X \ast \H)$ such that $b(g) = \xi_{\r(g)} - \pi(g)\xi_{\d(g)}$ for $\mu_\rG$-almost every $g\in \rG$. Then for $n \geq 1$, we can set $E_n = \{x \in X:\norm{\xi_x} \leq n\}$ and that makes $\bigcup_{n = 1}^\infty E_n = X$ true. Furthermore, we check that for $g \in \rG\vert_{E_n}$ we have 
$$
\norm{b(g)}\leq \norm{\xi_{\r(g)}}+\norm{\xi_{\d(g)}}\leq 2n < \infty.
$$

For the only if part, we assume the existence of the sequence of measurable subsets $\{E_n\}_{n = 1}^\infty$ that almost everywhere exhausts $X$ and $\sup\{\norm{b(g)}: g \in \rG\vert_{E_n}\} < \infty$ for each $n \geq 1$. Then by \cite[Lemma 3.21]{An05}, we know that for each $n$ we know that $b$ restricted to $\rG\vert_{E_n}$ is a coboundary and therefore we can find a section $\xi^n \in S(E_n \ast \H)$ with \begin{equation}\label{formulae}
    b(g) = \xi_{\r(g)}^n - \pi(g)\xi_{\d(g)}^n, 
\end{equation} for $\mu_\rG$-almost every $g\in\rG\vert_{E_n}$. But we notice that the formula \eqref{formulae} also defines $\xi^n$ in the saturation $F_n=\bigcup_{g\in \rG} \r(g(E_n))$, when reinterpreted as $\xi_{\r(g)}^n=b(g)+\pi(g)\xi_{\d(g)}^n$. This definition does not depend on the choice of $g$; if $g,h \in \rG\vert_{E_n}$ and $x=\r(g)=\r(h)$, then $(g^{-1},h)$ is a composable pair, $g^{-1}h$ lies in $\rG\vert_{E_n}$ and
\begin{align*}
[b(g)+\pi(g)\xi_{\d(g)}^n] -  [b(h)+\pi(h)\xi_{\d(h)}^n] &= -\pi(g)b(g^{-1}h) + \pi(g)\xi_{\d(g)}^n - \pi(h)\xi_{\d(h)}^n\\
 &= \pi(g)\big(-b(g^{-1}h) + \xi_{\d(g)}^n -\pi(g^{-1}h)\xi_{\d(h)}^n \big)\\
 &=0. 
\end{align*}

Thus we have a sequence of sections $\{\xi^n\}_{n = 1}^\infty$ satisfying \eqref{formulae} but for $\mu_\rG$-a.e. $g \in \rG\vert_{F_n}$. Now for $x \in F = \bigcup_{n = 1}^\infty F_n$, define 
\begin{align*}
\xi_x = \xi^{n_x}_x, \quad \text{where} \quad n_x = \min \{n \geq 1 : x \in F_n\},
\end{align*} and let $\xi_x=0$ for $x\not\in F$. Note that if $x,y\in X$ lie in the same orbit, then $n_x = n_y$ since each $F_n$ is invariant.
Thus $$b(g) = \xi_{\r(g)}^{n_{\r(g)}} - \pi(g)\xi_{\d(g)}^{n_{\d(g)}}=\xi_{\r(g)} - \pi(g)\xi_{\d(g)},$$ for all $g\in \rG\vert_F$.
\end{proof}

\begin{ex}\label{exstrongunbound}
    Suppose $G$ is a countable group acting in a measure-preserving way on the standard probability space $(X,\mu)$ and we form the transformation groupoid $\rG=X\rtimes_\theta G$ as in Example \ref{transformation}. Then any group representation $\pi:G\to\mathcal U(\mathcal H)$ gives rise to a representation $\pi_\rG$ of $\rG$ and any $1$-cocycle $b$ for $\pi$ gives a $1$-cocycle $b_\rG$ for $\pi_\rG$ as follows. In fact, $\pi_\rG$ represents $\rG$ on the Hilbert bundle $X\ast \mathcal K$ with constant fibers $\mathcal K_x=\mathcal H$, for all $x\in X$ and we have \begin{align*}
        \pi_\rG(g,x)&=\pi(g), \quad\forall (g,x)\in \rG \\
        b_\rG(g,x)&=b(g), \quad\forall (g,x)\in \rG.
    \end{align*} One can check that $\pi_\rG$ is (weak) mixing if $\pi$ is (weak) mixing and $b_\rG$ is (strongly) unbounded if $b$ is unbounded. Moreover, if $\pi \prec\rho$ for
another representation $\rho$ of $G$, then $\pi_\rG \prec\rho_\rG$ as well. When $\pi$ is the left regular representation, then $\pi_\rG$ is unitarily equivalent to the left regular representation of $\rG$. In particular, if $G$ admits a (strongly) unbounded $1$-cocycle into an orthogonal representation weakly contained in the left regular representation, then so does $\rG=X\rtimes_\theta G$.
\end{ex}

Hoff proves in \cite[Lemma 2.2]{Ho16} that, for equivalence relations, unboundedness coincides with strong unboundedness. His proof uses deep results of Feldman-Moore, which are simply not available for more general groupoids, leaving us with the question of whether this is true in general or not. We were able to prove that this is the case for a large class of groupoids, introduced before as semidirect product groupoids. This is the content of Proposition \ref{unbound}. With that objective in mind, let us first treat the subcase of group bundles.

\begin{lem}\label{gpbun}
    If $\rG = (G_x)_{x\in X}$ is a group bundle and $b$ is an unbounded $1$-cocycle, then $b$ is strongly unbounded.
\end{lem}

\begin{proof}
    By \cite[Lemma 3.21]{An05}, there is a measurable, non-null subset $E\sub X$ such that
    \[
        \sup_{g\in G_x} \norm{b(g)} = +\infty \quad \text{for all }x\in E.
    \]
    Let $\delta\coloneqq \mu(E)>0$. Given $R>0$, for each $x\in X$ there is some $g_x\in G_x$ with $\norm{g_x} \geq R$.

Let $ N:\rG\to [0,+\infty) $ denote the composition $ \rG \xrightarrow{b} X\ast \H \xrightarrow{\norm{\cdot}} [0,+\infty) $. Noting that $ N $ is Borel and the restriction of the domain map
\[
   {\rm  d}: N^{-1}((R,+\infty))\cap \rG\vert_{E} \to E
\]
is countable-to-one, by Lusin-Novikov \cite[Theorem 18.10]{Ke95} there is a countable collection of measurable subsets $ U_{n}\sub N^{-1}((R,+\infty))\cap \rG\vert_{E}  $ such that $ N^{-1}((R,+\infty))\cap \rG\vert_{E} =\bigcup_{n}U_{n} $ and each $ d:U_{n}\to E $ is a Borel isomorphism for $ n\in \N $. Note that each $U_n$ is in fact an element of $[\rG\vert_E]\sub [[\rG]]$. Hence, choosing some $ U = U_{n} $, we may extend $ U $ to an element $ \sigma\in [\rG] $.
    It follows that 
    \[
        \mu(\{x\in X: \norm{b(x\sigma)} \geq R\}) \geq \mu(E) = \delta >0,
    \] finishing the proof. \end{proof}

\begin{prop}\label{unbound}      
    Let $b$ be a 1-coboundary associated with the representation $\pi$ of $\rG$ in $X\ast \H$. If $b$ is strongly unbounded, then it is unbounded. For semidirect product groupoids, the two notions coincide.
\end{prop}

\begin{proof}
Let $\delta>0$ as in the definition of strong unboundedness and suppose that $b$ is bounded. Then there exists a measurable subset $E \subset X$ such that $\mu(E) > 1- \frac{\delta}{2}$ and $R = \sup \{\|b(g)\| : g\in \rG\vert_E\} < \infty.$. But by assumption, there must be a $\si\in[\rG]$ such that $F = \{x\in X: \|b(x\si)\| > R\}$ has measure at least $\delta$. Then we would have $\d\big((E\cap F)\si\big)\subset E^c$, which implies 
$$
\mu(\d\big((E\cap F)\si\big))=\mu(\r\big((E\cap F)\si\big))=\mu(E\cap F)\leq \mu(E^c).
$$ 
In consequence, 
\begin{equation*}
\mu(E\cup F)\geq\mu(E)+\mu(F)-\mu(E^c)=2\mu(E)+\mu(F)-1>1.
\end{equation*}

Now, if $b$ is unbounded and $\rG$ is isomorphic with a semidirect product groupoid, we can, without loss of generality, assume that $\rG=\Gamma\rtimes_{\delta}\mathcal{R}$, with $\mathcal{R}=\mathcal{R}_\rG$ and $\Gamma={\rm Iso}(\rG)$. Since $\rG$ contains copies of both $\Gamma$ and $\mathcal{R}$, the $1$-cocycle $b$ restricts to $1$-cocycles $b\vert_\Gamma$ and $b\vert_{\mathcal{R}}$, for the representations $\pi\vert_{\Gamma}$ and $\pi\vert_{\mathcal{R}}$, respectively. If any of these is strongly unbounded, the proof is finished since $[\Gamma],[\mathcal R]\subset [\rG]$ in a natural manner. Hence we assume that it is not the case. But then Lemma \ref{gpbun} and \cite[Lemma 2.1]{Ho16} imply that both $b\vert_\Gamma$ and $b\vert_{\mathcal{R}}$ are bounded. Then, by definition, there are sequences of measurable subsets $\{E_n\}_{n = 1}^\infty, \{G_n\}_{n = 1}^\infty$ of $X$ with $\mu(\bigcup_{n = 1}^\infty E_n) =\mu(\bigcup_{n = 1}^\infty G_n)= 1$ and the quantities $s_1=\sup\{\|b(g)\|: g \in {\mathcal R}\vert_{E_n}\}$, $s_2=\sup\{\|b(g)\|: g \in {\Gamma}\vert_{G_n}\}$ are finite for each $n \geq 1$. Define the double-indexed sequence of measurable sets 
$$
A_{n,m}=E_n\cap G_m
$$ 
and note that 
$$
\mu\big(\cup_{n,m = 1}^\infty A_{n,m}\big)\geq\mu\big((\cup_{n}^\infty E_n)\cap(\cup_{m = 1}^\infty G_{m})\big)=1,
$$ 
while, for $(g,y,x)\in \rG$ such that $y,x\in A_{n,m}$, one has $$\norm{b(g,y,x)}\leq \norm{b(g,x,x)}+\norm{b(\e_x,y,x)}\leq s_1+s_2<\infty$$ and so we conclude that $b$ is a bounded $1$-cocycle.
\end{proof}

In particular, it follows that all non-trivial $1$-cocycles on semidirect product groupoids are strongly unbounded.

\subsection{Malleable deformations}

In the following, we survey the basics of Sorin Popa's deformation rigidity theory as well as various relevant approaches/results from \textcite{dSHH21} which motivate this paper's main results.

The intuitive idea behind deformation/rigidity theory is to study rigidity results for a von Neumann algebra $ M $ which can be deformed inside another algebra $ \widetilde{M}\supseteq M $ by an action $ \alpha: \R \to \Aut(\widetilde{M}) $ whilst containing subalgebras which are \textit{rigid} with respect to the deformation.

Let $(M,\tau)$ be a tracial von Neumann algebra and $ \Aut(M) $ the group of trace-preserving $^*$-automorphisms of $M$. Then we have the following fundamental definition due to Popa.

\begin{defn}[Popa]
  Let $ \widetilde{M}\supseteq M $ be a trace-preserving inclusion of tracial von Neumann algebras.
  \begin{enumerate}
    \item A \textit{malleable deformation $ \alpha $ of $ M $ inside $ \widetilde{M} $} is a strongly-continuous action $ \alpha:\R\to \Aut(\widetilde{M}) $ such that $ \alpha_{t}(x)\xrightarrow{\norm{\cdot}_{2}}x $ as $ t\to 0 $ for every $ x\in \widetilde{M} $.
    \item An \textit{s-malleable deformation $ (\alpha,\beta) $ of $ M $ inside $ \widetilde{M} $ } is a malleable deformation $ \alpha $ combined with a distinguished involution $ \beta\in\Aut(\widetilde{M}) $ such that $ \beta\vert_{M} = id $ and $ \beta \alpha_{t} = \alpha_{-t} \beta $ for all $ t\in \R $.
  \end{enumerate}
\end{defn}

On its own, deformations do not give much information about the algebra itself; however, they do provide one with a quantitative way to locate subalgebras with prescribed properties that force them to be \textit{rigid} with respect to the deformation. Explicitly, given a malleable deformation $ \alpha $ of $ M $ inside $ \widetilde{M} $, a subalgebra $ Q\sub M $ is $ \alpha $-\textit{rigid} if the deformation converges uniformly on the unit ball of $ Q $, i.e.
$$
  \lim_{t\to 0}\sup_{x\in (Q)_{1}} \norm{\alpha_{t}(x)-x}_{2} =0.
$$

Locating $\alpha$-rigid subalgebras is a fundamental step to applying deformation/rigidity theory to a problem. An attempt can be made to directly prove rigidity by harnessing the structural properties of the subalgebra, such as the relative property ($T$). The recent notion of maximal rigidity for subalgebras studied by de Santiago, Hayes,
Hoff, and Sinclair in \cite{dSHH21} has had relative success in providing a more indirect way of proving rigidity. 

\begin{defn}[\textcite{dSHH21}]\label{ben}
  Let $ (\alpha,\beta) $ be an $ s $-malleable deformation $ M $ inside $ \widetilde{M} $ where $ M $ and $ \widetilde{M} $ are both assumed to be finite. Then an $ \alpha $-rigid subalgebra $ Q\sub M $ is \textit{maximal $ \alpha $-rigid} if whenever $ P\sub M $ is an $ \alpha $-rigid subalgebra containing $ Q $, it follows that $ P = Q $.
\end{defn}

\begin{defn}
  Let $ \alpha $ be a malleable deformation of $ M $ inside $ \widetilde{M} $ where $ M$, $\widetilde{M} $ are finite. Suppose that $ Q\sub M $ is an $ \alpha $-rigid subalgebra of $ M $. Then a subalgebra $ P\sub M $ is an \textit{$\alpha$-rigid envelope of $ Q $} if 
  \begin{itemize}
    \item $ P $ is $ \alpha $-rigid
    \item $ P\supseteq Q $
    \item if $ N \sub M $ is $ \alpha $-rigid and $ N\supseteq Q $, then $ N\sub P $.
  \end{itemize}
\end{defn}

Given a rigid subalgebra $Q\sub M$, one often wishes to know whether rigidity passes to some specific subalgebra $N\sub M$ which contains $Q$ (as is the case with various weak normalizers of $Q$). Allowing the concept of rigid envelopes, it suffices then to prove that this $N$ is contained in a rigid envelope of $Q$. This approach is demonstrated in \cite[Cor. 6.7]{dSHH21} and shows the relevance of identifying rigid envelopes.

One would be justified in being skeptical as to whether rigid envelopes even exist for given subalgebras. In fact, they do not exist in general; however, some of the main results in \textcite{dSHH21} show that they do in many natural cases. We shall use these results in crucial ways and thus sketch them here for reference.

\begin{thm}[\textcite{dSHH21}] \label{rigenvelope}
  Let $ (\alpha,\beta) $ be an $ s $-malleable deformation of tracial von Neumann algebras $ M\sub \widetilde{M} $. Then any $ \alpha $-rigid subalgebra $ Q\sub M $ with $ Q^{\prime}\cap \widetilde{M}\sub M $ is contained in a unique maximal $ \alpha $-rigid subalgebra $ P\sub M $.
\end{thm}

In the following, we will utilize Popa's transversality inequality \cite[Lemma 2.1]{Po08} and Popa's spectral gap argument \cite{Po07}. We also include the relevant definitions for Hilbert bimodules. For a presentation of this version of the argument, see \cite[Theorem 3.2]{Ho16}.

\begin{defn}
    Let $ N\subseteq M $ be a von Neumann subalgebra. An $M$-$M$ bimodule $ _M \H_M $ is said to be \textit{mixing relative to $ N $} if for any sequence $ (x_{n})_{n=1}^{\infty} $ in $ (M)_{1} $ such that $ \norm{\E_{N}(yx_{n}z)}_{2} \to 0 $ for every $ y,z\in M $, we have that 
    \[
        \lim_{n\to \infty} \sup_{y\in(M)_{1}}|\inp{x_{n} \xi y}{\eta}| = \lim_{n\to \infty} \sup_{y\in(M)_{1}}|\inp{y \xi x_{n} }{\eta}| = 0 \quad \text{ for all } \xi,\eta\in \H
    \]

\end{defn}

\begin{defn}
    An $ M $-$ N $ bimodule $ _M\H_N $ is said to be \textit{weakly contained} in an $ M $-$ N $ bimodule $ _M\K_N $, written $ _M\H_N \prec _M\K_N$, if for any $ \varepsilon>0 $, finite subsets $ F_{1} \sub M $, $ F_{2}\sub N $, and $ \xi\in \H $, there are $ \eta_{1},\ldots, \eta_{n}\in \K $ such that 
    \[
        |\inp{x \xi y}{\xi} - \sum_{j=1}^{n}\inp{x \eta_{j}y }{\eta_{j}} | < \varepsilon \quad \text{for all } x\in F_{1}, y\in F_{2}
    \]

\end{defn}

\begin{lem}[Popa's transversality inequality] \label{transv}
Let $\alpha: \R \to \Aut(\widetilde M)$, $\beta \in \Aut(\widetilde M)$ be an $s$-malleable deformation of $M \subset \widetilde M$. Set $\delta_t(x) = \alpha_t(x) - \E_M(\alpha_t(x))$ for $x \in M$. Then for all $x, y \in M$ and $t \in \R$, \begin{enumerate}
    \item[(i)] $\norm{\delta_{2t}(x)}_2 \leq 2\norm{\alpha_{2t}(x) - x}_2 \leq 4\norm{\delta_t(x)}_2$, and 
    \item[(ii)] $\norm{[\delta_t(x), y]}_2 \leq 2\norm{x}\norm{\alpha_t(y) - y}_2 + \norm{[x, y]}_2.$
\end{enumerate}
\end{lem}
\begin{proof} For (i), we see that \begin{align*}
\norm{\delta_t(x)}_2 &\leq \norm{\alpha_t(x) - x}_2 + \norm{x - \E_M(\alpha_t(x))}_2 \\
&=\norm{\alpha_t(x) - x}_2 + \norm{\E_M(x - \alpha_t(x))}_2 \\
&\leq 2\norm{\alpha_t(x) - x}_2,
\end{align*}
and since $\beta \alpha_t = \alpha_{-t}\beta$ and $\beta|_M = {\rm id}$, we have
\begin{align*}
\norm{\alpha_{2t}(x) - x}_2 
&= \norm{\alpha_t(x) - \alpha_{-t}(x)}_2 
\leq \norm{\alpha_t(x) - \E_M(\alpha_t(x))}_2 + \norm{\alpha_{-t}(x) - \E_M(\alpha_t(x))}_2 \\
&= \norm{\delta_t(x)}_2 + \norm{\beta(\alpha_t(x) - \E_M(\alpha_t(x)))}_2
= 2\norm{\delta_t(x)}_2.
\end{align*}
For (2), 
\begin{align*} 
\norm{[\delta_t(x), y]}_2 
& = \norm{(1-E_M)([\alpha_t(x), y])}_2 
\leq \norm{[\alpha_t(x), y]}_2 \\
&\leq \norm{\alpha_t(x)y - \alpha_t(x)\alpha_t(y)}_2 + \norm{[\alpha_t(x), \alpha_t(y)]}_2 + \norm{\alpha_t(y)\alpha_t(x) - y\alpha_t(x)}_2  \\
&\leq 2\norm{x}\norm{\alpha_t(y) - y}_2 + \norm{[x, y]}_2,  
\end{align*} as claimed. \end{proof}

\begin{thm}[Popa's spectral gap argument \cite{Po07}]\label{popaspectralgap}
Let $ (\alpha,\beta) $ be an $ s $-malleable deformation of tracial von Neumann algebras $ M\sub \widetilde{M} $ and assume that $M$ has no amenable direct summands. Suppose further that the orthocomplement bimodule $_{M} L^2(\widetilde{M}) \ominus L^2(M)_M $ is weakly contained in the coarse $M$-$M$ bimodule and mixing relative to an abelian subalgebra $A\sub M$. Then there is a central projection $z\in Z(M)$ such that $M(1-z)$ is prime and 
    \[
    \lim_{t\to0}\sup_{x\in (Mz)_{1}} \norm{\alpha_{t}(x)-x}_{2} =0.
    \]
    In particular, if the convergence $\alpha_t \to{\rm id}$ as $t\to 0$ is not uniform on the unit ball of $M$, then $M$ is not the join of any pair of commuting type $\mathrm{II}$ von Neumann subalgebras.  
    
    
\end{thm}

\begin{proof}
    Most of the details of this proof are contained in \cite[Theorem 3.2]{Ho16}; however, we have a slightly modified conclusion.

    Suppose that $M\cong N\,\vee Q$ for type $\mathrm{II}$ subalgebras $N,Q\leq M$ with $[N,Q] = 0$. As $M$ is nowhere amenable, we may assume without loss of generality that $Q$ is also nowhere amenable. Then Hoff's proof gives that the deformation must be rigid on $(N)_1$. As $N$ commutes with $Q$, it is contained in the normalizer of $Q$. Hence, appealing to \cite[Corollary 6.7(ii)]{dSHH21}, it follows that the convergence is rigid on both $(Q)_1$ and $(N)_1$, whence it is rigid on all of $(M)_1$. \end{proof}

\section{Gaussian extensions and the associated \texorpdfstring{$s$}--malleable deformations} \label{gaussianc}

In this section we construct the s-malleable deformation that will be used to prove the main results. We recall once more that Gaussian actions have been used to construct $s$-malleable deformation for group von Neumann algebras in \cite{dSHH21, PeSi12, Si11}, crossed products in \cite{Va13} and for equivalence relation von Neumann algebras in \cite{Ho16}. We will follow some of the same reasoning in our wider context. 

\subsection{The Gaussian extension of \texorpdfstring{$\rG$}-} 
Let $ \rG $ be a discrete measured groupoid, $\pi$ an orthogonal representation of $\rG$ on a real Hilbert bundle $X \ast \H$, and let $\{\xi^n\}_{n = 1}^\infty$ be an orthonormal fundamental sequence of sections for $X \ast \H$. For $x\in X$, we consider the measure space
\[
    (\Omega_x, \nu_x) = \prod_{i = 1}^{\dim \H_x} (\R, \tfrac{1}{\sqrt{2\pi}}e^{-s^2/2}ds),
\]
and define $\omega_x: \spn_\R (\{\xi_x^n\}_{n = 1}^{\dim \H_x}) \to \U(L^\infty(\Omega_x))$ by
\[
    \omega_x\left(\sum_{n = 1}^{\dim \H_x} a_n\xi_x^{n}\right) = \exp\left({i\sqrt{2}\sum_{n = 1}^{\dim \H_x} a_nS_x^{n}}\right)
\] 
where $S^n_x$ is the $n$th-coordinate function $S_x^n((s_i)_{i = 1}^{\dim \H_x}) = s_n$ for $1\le n \le \dim \H_x$.

Then $\omega_x$ extends to a $\norm{\cdot}_{\H_x} - \norm{\cdot}_2$ continuous map $\omega_x : \H_x \to \U(L^\infty(\Omega_x))$ satisfying 
\begin{equation} \label{axioms}
\tau(\omega_x(\xi)) = e^{-\|\xi\|^2_{\H_x}}, \quad \omega_x(\xi + \eta) = \omega_x(\xi)\omega_x(\eta), \quad \omega_x(-\xi) = \omega_x(\xi)^*, \quad \forall\xi, \eta \in\H_x.
\end{equation}

For $x \in X$, one also has $D_x = \spn_\CC (\{\omega_x(\xi)\}_{\xi \in \H_x})\subset L^\infty(\Omega_x)$ is WOT-dense in $L^\infty(\Omega_x)$. Now for every $g\in\rG$, define a $^*$-homomorphism $\rho(g): D_{\d(g)} \to L^\infty(\Omega_{\r(g)})$ by 
$$
\rho(g)\omega_{\d(g)}(\xi) = \omega_{\r(g)}(\pi(g)\xi),
$$
which is well defined and $\norm{\cdot}_2$-isometric since \eqref{axioms} implies 
$$
\tau(\omega_{\d(g)}(\eta)^*\omega_{\d(g)}(\xi)) = \tau(\omega_{\r(g)}(\pi(g)\eta)^*\omega_{\r(g)}(\pi(g)\xi)) \quad \forall\xi, \eta \in \H_{\d(g)}.
$$
So $\rho(g)$ extends to a trace-preserving $^*$-isomorphism $\rho(g): L^\infty(\Omega_{\d(g)}) \to L^\infty(\Omega_{\r(g)})$.  
Let $\theta_{g}: \Omega_{\d(g)} \to \Omega_{\r(g)}$ be the induced measure space isomorphism such that $\rho(g)\phi = \phi\circ \theta_{g}^{-1}$ for all $\phi \in L^\infty(\Omega_{\r(g)})$. 

So far, this construction provides a measurable $X$-bundle of commutative von Neumann algebras $\mathcal B=\{L^\infty(\Omega_{x})\}_{x\in X}$ and note that the maps $\{\rho(g)\}_{g\in \rG}$ give us an action of $\rG$ on $\mathcal B$ and hence an action of $\rG$ on the underlying space. Then the natural thing to do is to use this action to provide an extension of $\rG$ by mimicking the transformation groupoid construction. We will then define an $s$-malleable deformation of $L(\rG)$ inside the algebra of the bigger groupoid.

We consider $\widetilde X\equiv X \ast \Omega=\bigsqcup_{x\in X}\{x\}\times \Omega_x$ as a measurable bundle with $\sigma$-algebra generated by the maps $(x, r) \mapsto \omega_x(\sum_{i \in I}a_i\xi^i_x)(r)$ for $I \subset \N$ finite and $a_i \in \R$. The natural measure $\mu \ast \nu$ on $X \ast \Omega$ is given by 
$$
[\mu \ast \nu](E) = \int_X \nu_x(E_x)d\mu(x),
$$
where $E_x = \{s \in \Omega_x: (x, s) \in E\}$. We define the \emph{Gaussian extension of $\rG$} to be the transformation groupoid $\widetilde \rG=(X \ast \Omega)\rtimes_\theta \rG$, explicitly given by 
\begin{itemize}
    \item As a set, $\widetilde \rG=\{(g,x,r)\in \rG\times  (X \ast \Omega)  : \r(g)=x\}$. The unit space is identified with $X \ast \Omega$
    \item The groupoid operations are 
    \begin{align*}
        \r(g,x,r)=(x,r), \quad &\quad\d(g,x,r)=(\d(g),\theta_g^{-1}(r)) \\
        (g,x,r)(h,\d(g),\theta_g^{-1}(r))=(gh,x,r), \quad&\quad (g,x,r)^{-1}=\big(g^{-1},\d(g),\theta_g^{-1}(r)\big)
    \end{align*}
    \item $\widetilde \rG$ inherits a natural measurable structure as a subset of the product $\rG\times  (X \ast \Omega) $. Lastly, $\mu*\nu$ plays the role of the invariant probability measure on $X \ast \Omega$.
\end{itemize}

Then we note that $L(\widetilde \rG)$ contains copies of $L^\infty(\widetilde X,\mu*\nu) $ and $L(\rG)$ such that 
\[
    L(\widetilde \rG)=  \vN{L^\infty(X \ast \Omega), \{u_\sigma\}_{\sigma\in[\rG]}}=\vN{L^\infty(X \ast \Omega), L(\rG)} \subset \B(L^2(\widetilde \rG)).
\]
and we have the relation  $$u_\sigma \phi u_\sigma^*=\theta_\si(\phi),$$ for all $\sigma\in [\rG], \phi\in L^\infty(X\ast \Omega) $. Here $\theta_\si(\phi)$ denotes, naturally, the induced action of $[\rG]$ on $L^\infty(X\ast \Omega)$.

\subsection{The associated \texorpdfstring{$s$}--malleable deformation of \texorpdfstring{$L(\rG)$}-}

Let $ b $ be a $1$-cocycle for the representation $ \pi $ on $ X*\H $. Set $ M = L(\rG) $, $ \widetilde{M}\coloneqq L(\widetilde{\rG}) $. For $ t\in\R $, define $ c_{t}:\widetilde{\rG}\to \mathbb S^{1} $ by
\[
  c_{t}(g,x,r) = \omega_{x}(tb(g))(r),
\] 
which is a multiplicative function (a homomorphism) and it satisfies 
\[
    c_{t_1+t_2}(g,x,r)=c_{t_1}(g,x,r)c_{t_2}(g,x,r).
\] 
Indeed, we verify the first assertion. Given $(g,x,r),(h,\d(g),\theta_g^{-1}(r))\in \widetilde{\rG}$, we see that
\begin{align*}
    c_t(gh,x,r)&=\omega_{x}(tb(g))(r)\omega_{x}(t\pi(g)b(h))(r) \\
&=\omega_{x}(tb(g))(r)\big[\rho(g)\omega_{\d(g)}(tb(h))\big](r) \\
&=\omega_{x}(tb(g))(r)\omega_{\d(g)}(tb(h))\big(\theta_g^{-1}(r)\big) \\
&=c_t(g,x,r)c_t(h,\d(g),\theta_g^{-1}(r)).
\end{align*} 
And this leads us to define, for $ t\in\R $ and $ \sigma\in [\widetilde{\rG}] $, the function $ f_{c_{t},\si}\in \mathcal{U}(L^{\infty}(X\ast \Omega)) $ given by
$$
  f_{c_{t},\sigma}(x,r) = \omega_{x}(tb(x\sigma))(r)=c_t(x\sigma, x,r), 
$$ so that we obtain an SOT-continuous $\mathbb R$-action $ \alpha_{b,t}\in \Aut(\widetilde{M}) $ by using the formula
$$
  \alpha_{b,t}(au_{\sigma}) = f_{c_{t},\sigma}au_{\sigma}.
$$ and extending by linearity. 

Now we compute
\begin{align*}
  \tau(f_{c_{t},\sigma}) &= \int_{X\ast \Omega} f_{c_{t},\sigma} \dd{\mu\ast \nu} = \int_{X} \int_{\Omega_{x}} f_{c_{t},\sigma}(x,r) \dd{\nu_{x}}(r) \dd{\mu(x)} \\
  &= \int_{X}\int_{\Omega_{x}} \omega_{x}(tb(x\sigma))(r) \dd{\nu_{x}}(r) \dd{\mu(x)} \\
  &=\int_{X} e^{-t^2\norm{b(x\sigma)}^{2}}  \dd{\mu(x)}.
\end{align*} So
\begin{align*}
  \norm{\alpha_{b,t}(au_{\sigma})-au_{\sigma}}_{2}^{2} &= \norm{f_{c_{t},\sigma}au_{\sigma}-au_{\sigma}}_{2}^2 \leq \norm{a}^{2} \norm{f_{c_{t},\sigma}-1}_{2}^{2}  = 2\norm{a}^{2}(1-\Re \tau(f_{c_{t},\sigma})) \\
  &= 2\norm{a}^{2}\lr{1-\int_{X} e^{-t^2\norm{b(x\sigma)}^{2}}\dd{\mu(x)}}\xrightarrow{t\to0}0.
\end{align*}
Moreover, if $b$ is bounded, the convergence is uniform. Next, note that defining $\beta_x(\omega_x(\xi))=\omega_x(-\xi)=\omega_x(\xi)^*$ for $x\in X$ gives a $^*$-automorphism of $L^\infty(\Omega_x)$, which leads to $\beta\in\Aut(L^\infty(X\ast \Omega)) $ defined by $\beta(a)(x,r)=\beta_x(a(x,\cdot))(r)$, for $a\in L^\infty(X\ast \Omega)$.

\subsection{A Fock space realization}

Let us first analyze the ``transition'' maps $ \rho(g):L^{\infty}(\Omega_{\d(g)}) \to L^{\infty}(\Omega_{
(g)}) $ utilized in the construction of the Gaussian extension of $ \rG $.

Since each fiber $ \Omega_{x} $ is a finite measure space, we have that $ \cls{L^{\infty}(\Omega_{x})}^{\norm{\cdot}_{2}} = L^{2}(\Omega_{x}) $, thus we may extend $ \rho(g) $ to an isometry $ \rho(g):L^{2}(\Omega_{\d(g)})\to L^{2}(\Omega_{
(g)}) $. Finally, after restricting, we obtain a unitary map
$$
  \rho(g): L^{2}(\Omega_{\d(g)})\ominus \C \to L^{2}(\Omega_{
(g)})\ominus \C
$$

Now form the Hilbert bundle $ X\ast \K $ as follows
\begin{itemize}
  \item $ \K_{x}\coloneqq L^{2}(\Omega_{x})\ominus \C $ for $ x\in X $
  \item $ \sigma $-algebra determined by fundamental sections $ \omega_{0}(\Span_{\Q}\{\xi^{n}\}_{n=1}^{\infty}) $, where $ \{\xi^{n}\}_{n=1}^{\infty} $ as before and 
    $$
    [\omega_{0}(\eta)](x) = \omega_{x}(\eta(x)) - \tau(\omega_{x}(\eta(x))) = \omega_{x}(\eta(x)) - e^{-\norm{\eta(x)}^{2}} \text{ for }\eta\in S(X\ast\H).
    $$
\end{itemize}

As $ \rho(gh) = \rho(g) \rho(h) $ for all $ (g,h)\in \rG^{(2)} $, we see that $ \rho $ is a representation of $ \rG $ on $ X\ast \K $. In fact, $\rho$ corresponds to the \emph{Koopman representation} $\kappa_0$ associated with the action of $\rG$ on $ X\ast \K $.

\begin{lem}\label{fockspace}
  For each $ x\in X $, let $ \widehat{\H}_{x} = \bigoplus_{n=1}^{\infty}(\H_{x}\otimes_{\R}\C)^{\odot n} $. Then the representation $ \rho $ of $ \rG $ on $ X\ast \K $ is unitarily equivalent to the representation $ \widetilde{\pi} = \oplus_{n=1}^{\infty} \pi_{\C}^{\odot n} $ of $ \rG $ on $ X\ast \widehat{\H} $.
\end{lem}

\begin{proof}
  For $ x\in X $, set $ U_{x}:D_{x}\to \C\oplus \widehat{\H}_{x} $ by $ \omega_{x}(\xi)\mapsto e^{-\norm{\xi}^{2}}\bigoplus_{n=0}^{\infty}\frac{(i \sqrt{2} \xi)^{\odot n}}{n!} $ for $ \xi\in \H_{x} $. Note that $ U_{x} $ is well-defined and isometric as 
  $$
    \inp{e^{-\norm{\xi}^{2}}\bigoplus_{n=0}^{\infty}\frac{(i \sqrt{2} \xi)^{\odot n}}{n!}}{e^{-\norm{\eta}^{2}}\bigoplus_{n=0}^{\infty}\frac{(i \sqrt{2} \xi)^{\odot n}}{n!}} = \tau(\omega_{x}(\eta)^{*} \omega_{x}(\xi)).
  $$


  Note that $ \C\sub U_{x}(D_{x}) $ as $ z \cdot\omega_{x}(0)\mapsto z $ for all $ z\in \C $. Moreover, one can inductively check via a derivative argument that $ \xi_{1}\odot \cdots \odot \xi_{n} \in \cls{U_{x}(D_{x})} $ for all $ \xi_{i}\in \H_{x} $ and $ n\in \N $. Thus, we can extend $ U_{x} $ to a unitary $ U_{x}: L^{2}(\Omega_{x})\to \C\oplus \widehat{\H}_{x} $.

  Now fix $g\in \rG$. Then, for $\xi\in \H_{\d(g)}$, we have that 
  \[
    [\rm{id}_{\C}\oplus \widehat{\pi}](g) U_{\d(g)} [\omega_{\d(g)}(\xi)] =  U_{\d(g)} \omega_{
\r(g)}(\pi(g)(\xi)) = U_{\d(g)} \rho(g)[\omega_{\d(g)}(\xi)].
  \]
    And so by density we have that for all $f\in L^2(\Omega_{\d(g)})$,
    \[
        [\rm{id}_{\C}\oplus \widehat{\pi}](g) U_{\d(g)} f =U_{\d(g)} \rho(g)f.
    \]
    As $U_{\d(g)}$ fixes $\C$, upon restricting to the orthocomplement we obtain a unitary equivalence.
\end{proof}

\begin{rem}\label{Gaussmix}
    It follows from Lemma \ref{fockspace} that any property of representations that is stable under tensor products passes from $\pi$ to $\rho$. This is the case with weak mixing \cite[Corollary 3.17]{GaLu17} (at least when the groupoid is ergodic). It also follows that if for some subgroupoid $\S\leq \rG$, the representation $\pi|_\S$ is weak mixing, then the action of $\S$ on $X*\Omega$ is weak mixing.
\end{rem}

\section{Bimodules induced by groupoid representations}\label{bimoddddd}

We now proceed to describe the process of inducing $ L(\rG) $-$ L(\rG) $ bimodules from representations of $ \rG $. Let $ M = L(\rG) $ and $ A = L^{\infty}(X) \sub M $. Then a representation $ \pi $ of $ \rG $ on a Hilbert bundle $ X\ast \H $ induces a group representation $ \widetilde{\pi}: [\rG]\to \U\lr{\int_{X}^{\oplus}\H_{x} \dd{\mu(x)}} $ by 

\[
  (\widetilde{\pi}(\sigma) \xi)(x) = \pi(x\sigma)  \xi_{\d(x\si)}
\]
for $ \sigma\in [\rG] $, $ x\in X$. Utilizing Connes fusion over $ A $, we may form the $ A $-$ L(\rG) $ bimodule
\[
  \B(\pi) \coloneqq \left[ \int_{X}^{\oplus}\H_{x}\dd{\mu(x)}\right] \otimes_{A} L^{2}(\rG).
\]
We wish to incorporate the representation $ \pi $ to upgrade $ \B(\pi) $ to an $ M $-$ M $ bimodule.

\begin{prop}\label{reptobim}
    The Hilbert space $ \B(\pi) $ has an $ L(\rG) $-$ L(\rG) $ bimodule structure such that 
    \[
        au_{\sigma}\cdot (\xi \otimes \eta) \cdot x = \widetilde{\pi}(\sigma)(\xi)\otimes au_{\sigma}\eta x
    \]
    for $ a\in A $, $ \sigma\in [\rG] $, $ \xi\in \int_{X}^{\oplus}\H_{x} \dd{\mu(x)} $, $ x\in M $, $ \eta\in L^{2}(\rG) $.
\end{prop}

\begin{proof}
    This can be proven exactly as in \cite[Lemma 5.3]{Ho16}.
\end{proof}

To fix notation, for a Hilbert bundle $ X\ast \H $ and $ \xi\in \int_{X}^{\oplus}\H_{x}\dd{\mu(x)} $, define $ \norm{\xi}_{\infty}:=\esssup_{x\in X} \norm{\xi_{x}} $. If $ \xi,\eta\in \int_{X}^{\oplus}\H_{x}\dd{\mu(x)} $ and $ \norm{\xi}_{\infty}, \norm{\eta}_{\infty}<+\infty $, define $ \inp{\xi}{\eta}_{A}\in A $ by $ \inp{\xi}{\eta}_{A}(x):= \inp{\xi_{x}}{\eta_{x}} $. Note that 
$$
\norm{\xi_x}^2 = \inp{\xi_x}{\xi_x} = \inp{\xi}{\xi}_A(x).
$$ 
Under this notation, the inner product in the aforementioned Connes fusion bimodule is given by
\[
    \inp{\xi \otimes \hat{m}}{\eta \otimes \hat{n}} = \inp{\inp{\xi}{\eta}_A\cdot \hat{m}}{\hat{n}}\quad\text{for }m,n\in M,\, \xi,\eta\in\int_X^{\oplus}\H_x \dd{\mu(x)}\text{ with } \norm{\xi}_\infty, \norm{\eta}_\infty <+\infty.
\]

A key tool in studying unitary representations of groups is the notion of positive-definite functions and multipliers. We have a corresponding definition for groupoids.

\begin{defn}
    Let $ \mathbb{F}\in \{\R,\C\} $. A Borel function $ \phi:\rG\to \mathbb{F} $ is \emph{positive-definite} if for a.e. $ x\in X $, for all $ n\in\N $ and $ g_{1},\ldots, g_{n}\in \rG^{x} $, the matrix $ [\phi(g_{i}^{-1}g_{j})]_{i,j} $ is positive semidefinite. The function is unital if $ \phi(x) = 1 $ for a.e. $ x\in X $.
\end{defn}

Positive-definite functions naturally arise from groupoid representations. Given a groupoid representation $ \pi:\rG\to \U(X\ast \H) $ and a fixed section $ \xi\in \int_{X}^{\oplus}\H_{x}\dd{\mu(x)} $ with $ \norm{\xi}_{\infty} < +\infty $, the function $ \phi_{\xi}:\rG\to \C $ by $ \phi_{\xi}(g) = \inp{\pi(g) \xi_{\d(g)}}{\xi_{\r(g)}} $ is positive-definite and is unital if $ \norm{\xi_{x}}=1$ for a.e. $ x\in X $ (see \cite{Jo05,AD13,Ki17}). Positive-definite functions also naturally give rise to completely-positive maps on $ L(\rG) $. 

\begin{prop}[\cite{Lu18}]
   Let $ \phi:\rG\to \C $ be a positive-definite function on $ \rG $. Then there is a unique completely-positive map $ \Phi:L(\rG)\to L(\rG) $ such that 
    \[
        \Phi(au_{\sigma}) = \phi(\cdot\, \sigma) au_{\sigma} \quad \text{for all } a\in A,\, \sigma\in[\rG].
    \]
    Moreover, $ \Phi $ is unital if $ \phi $ is unital.
\end{prop}
When the positive-definite function in question is the $ \phi_{\xi} $ coming from a groupoid representation and a section $ \xi $, we denote the corresponding completely-positive map $ \Phi_{\xi} $. With this notation, we have 
\[
    \Phi_{\xi}(au_{\sigma}) = \phi_{\xi}(au_{\sigma}) = \inp{\pi(\cdot \sigma) \xi_{\d(\cdot \sigma)}}{\xi_{\cdot}}au_{\sigma} = \inp{\widetilde{\pi}(\sigma) \xi}{\xi}_{A}au_{\sigma}.
\]
We record below some useful facts about how the completely positive map $\Phi_\xi$ interacts with the trace on $L(\rG)$. 
\begin{lem}\label{lem:phicringe}
    Suppose $ \xi_{1},\ldots, \xi_{s}\in \int_{X}^{\oplus}\H_{x}\dd{\mu(x)} $ with $ \norm{\xi_{i}}_\infty<+\infty $ for all $1\leq i \leq s$. Then the following hold.
    \begin{enumerate}
        \item[(a)] For any $ x,\psi_{1},\psi_{2}\in L(\rG) $, we have
            \[
                \sum_{i=1}^{s} \inp{x\cdot \xi_{i}\otimes \widehat{\psi_{1}}}{\xi_{i}\otimes \widehat{\psi_{2}}} = \sum_{i=1}^{s}\tau(\psi_{2} \psi_{1}^{\ast} \Phi_{\xi_{i}}(x)).
            \]
        \item[(b)] $ \norm{\sum_{i=1}^{s} \Phi_{\xi_{i}}(1)} = \norm{\sum_{i=1}^{s} \inp{\xi_{i}}{\xi_{i}}_{A}}_{\infty} $ 
        \item[(c)] $ \tau\circ\sum_{i=1}^{s}\Phi_{\xi_{i}} \leq \norm{\sum_{i=1}^{s} \inp{\xi_{i}}{\xi_{i}}_{A}}_{\infty}\cdot \tau $
        \item[(d)] The map $\sum_{i=1}^{s}\Phi_{\xi_i}$ is $L^2$-bounded, i.e. for all $x\in L(\rG)$, \[
            \norm{\sum_{i=1}^{s}\Phi_{\xi_i}(x)}_2 \leq \norm{\sum_{i=1}^{s} \inp{\xi_{i}}{\xi_{i}}_{A}}_{\infty} \cdot\norm{x}_{2}
        \]
    \end{enumerate}
\end{lem}

\begin{proof}

Items (a) and (b) follow directly from the definitions of the Connes-Fusion tensor product and of $ \Phi_{\xi} $, so we only show item (c). Suppose $ a\in A $ and $ \sigma\in [\rG] $. Then we compute 
\begin{align*}
    \tau(\Phi_{\xi}(au_{\sigma})) &= \int_X \E_{A}(\inp{\pi(\cdot \sigma)\xi_{\d(\cdot \sigma)}}{\xi_{\cdot}}au_{\sigma}) \dd{\mu(x)}\\
    &= \int_X \inp{\pi(x \sigma)\xi_{\d(x \sigma)}}{\xi_{x}} a(x) 1_{\{z\in X: z \sigma  = z\}}(x)\dd{\mu(x)}\\
    &= \int_X \inp{\xi}{\xi}_{A}(x) a(x) 1_{\{z\in X: z \sigma  = z\}}(x)\dd{\mu(x)} = \tau(\inp{\xi}{\xi}_{A}au_{\sigma})
\end{align*}
By linearity and normality, it follows that $ \tau(\Phi_{\xi}(x)) = \tau(x\inp{\xi}{\xi}_{A}) $ for all $ x\in L(\rG) $. Now suppose that $ x\in L(\rG)_{+} $. Then we apply the above computation and Cauchy-Schwarz to see 
\begin{align*}
    \tau\lr{\sum_{i=1}^{s}\Phi_{\xi_{i}}(x)} &= \tau(x^{1/2}\sum_{i=1}^{s}\inp{\xi_{i}}{\xi_{i}}_{A} x^{1/2}) = \inp{x^{1/2}}{\sum_{i=1}^{s}\inp{\xi_{i}}{\xi_{i}}x^{1/2}} \\
    &\leq \norm{x^{1/2}}_{2} \norm{\sum_{i=1}^{s}\inp{\xi_{i}}{\xi_{i}}}_{\infty} \norm{x^{1/2}}_2 = \norm{\sum_{i=1}^{s}\inp{\xi_{i}}{\xi_{i}}}_{\infty} \cdot\tau(x). \\
\end{align*}
Part (d) then follows from part (c) by applying the Kadison-Schwarz inequality (see \cite[Proposition 3.3]{Pa02}) to the map $\sum_{i=1}^{s}\Phi_{\xi_i}$. \end{proof}

\begin{prop}\label{prop:bimmixing}
    If $ \pi:\rG\to \U(X\ast \H) $ is a mixing $ \rG $-representation, then the bimodule $ _M \B(\pi)_M $ is mixing relative to $ A $.
\end{prop}

\begin{proof}
    Fix $ \xi \in S_{1}(X\ast\H) $ and $ \eps>0 $. The map $ \phi=\phi_\xi:\rG\to \C $ given by $ \phi(g) = \inp{\pi(g) \xi_{\d(g)}}{\xi_{\r(g)}} $ is then a unital positive-definite function on $ \rG $, whence there is a unique unital, completely-positive, $ \E_{A} $-preserving, $ A $-$ A $ bimodular map $ \Phi=\Phi_\xi:M\to M $ such that 
    \[
        \Phi(u_{\sigma}) = \phi(\cdot\, \sigma) u_{\sigma} \quad \text{for all } \sigma\in[\rG].
    \]
    By \cite[
Proposition 5.3]{AD13}, such a map extends to a unique contraction $ \widehat{\Phi}: L^{2}(M)\to L^2(M) $; moreover, the expectation-preservation and $ A $-$ A $-bimodularity imply that $ \widehat{\Phi}\in \inp{M}{A}\cap A^{\prime} $ (where $\inp{M}{A}$ denotes the Jones basic construction of $A\sub M$).

    Appealing to the mixingness of $ \pi $, for each $ k\in \N $ there is some measurable $ E_{k}\sub X $ with $ \mu(X\setminus E_{k})  <  2^{-k} $ such that  
    \[
      \mu_{\rG}\lr{\{g\in \rG\vert_{E_{k}}: | \phi(g)| >\eps\}} < +\infty.
    \]
    The projections $q_i:=1_{E_i}$ in $A$ satisfy $ \tau(1-q_{k}) \xrightarrow{k\to\infty} 0$.
    
    Consider the map $ \Phi_{k}: q_{k}M q_{k}\to q_{k}Mq_{k} $ given by $ \Phi_{k}(\cdot) = \Phi(q_{k}\cdot q_{k}) $. By \cite[Lemma 5.5]{AD13}, we have that $ \widehat{\Phi}_{k}\in \K(\inp{q_{k}M q_{k}}{q_{k}A q_{k}}) $ (i.e. the compact ideal space of the semifinite algebra $\inp{q_k M q_k}{q_k A q_k}$ as in \cite{Po06}).

    Suppose that $ (u_{n})_{n=1}^{\infty}$ is a sequence in $ (M)_{1} $ such that $ \norm{\E_{A}(m_{1}u_{n}m_{2})}_{2}\to 0 $ for all $ m_{1},m_{2}\in M $.
    
    Fix $ k\in \N $ such that $ \norm{1-q_{k}}_{2}^2 < \eps/6$. Then, for all $ n\in \N $, we have
    \begin{align*}
      \norm{u_{n}}_{2}^2 &= \norm{(1-q_{k})u_{n}(1-q_{k})}_{2}^2 + \norm{(1-q_{k})u_{n}^2 q_{k}}_{2}^2 + \norm{q_{k}u_{n}(1-q_{k})}_{2}^2 + \norm{q_{k}u_{n}q_{k}}_{2}^2\\
      &\leq 3\norm{1-q_{k}}_{2}^2 + \norm{q_{k}u_{n}q_{k}}_{2}^2  < \frac{\eps}{2} + \norm{q_{k}u_{n}q_{k}}_{2}^2.
    \end{align*}
    Note that, as the choice of $k$ is independent of $n\in \N$, the above estimate implies that $u_n$ is uniformly close to $q_k u_n q_k$.

    As the sequence $ (q_{k}u_{n}q_{k})_{n=1}^{\infty} $ in $ (q_{k}Mq_{k})_{1} $ satisfies the hypotheses of \cite[Prop. 1.3.3.5]{Po06}, the compactness of $ \widehat{\Phi}_{k} $ gives that 
    \[
      \norm{\Phi_{k}(q_{k}u_{n}q_{k})}_{2} \xrightarrow{n\to\infty} 0.
    \]
    Choose $ N\in \N $ such that $ \norm{\Phi_{k}(q_{k}u_{n}q_{k})}_{2} < \eps/2 $ for all $ n\geq N $. Then by Lemma \ref{lem:phicringe}(d),
    \begin{align*}
        \norm{\Phi(u_n)}_2 &\leq \norm{\Phi(u_n - q_k u_n q_k)}_2 + \norm{\Phi_k(q_k u_n q_k)}_2 \\
        &\leq \norm{u_n-q_k u_n q_k}_2 + \frac{\eps}{2} < \eps
    \end{align*}
    for all $ n\geq N $, so $ \norm{\Phi(u_{n})}_{2}\to 0 $.

    Now fix a countable subset $ \Gamma \sub [\rG]$ which generates $ \rG $ and write $ u_{n} = \sum_{\sigma\in \Gamma} a_{\sigma}^{n} u_{\sigma} $. For any $ m, y\in (M)_{1} $, observe that
    \begin{align*}
      \inp{u_{n}\cdot( \xi \otimes \widehat{m})\cdot y}{\xi \otimes \widehat{m}}&= \sum_{\sigma\in \Gamma} \inp{a_{\sigma}^{n}u_{\sigma}\cdot( \xi \otimes \widehat{m}) \cdot y}{\xi \otimes \widehat{m}} \\
      &= \sum_{\sigma\in \Gamma} \int_{X} \E_{A} (a_{\sigma}^{n}u_{\sigma}mym^{*}) \inp{\pi(x \sigma) \xi_{d(x \sigma)}}{\xi_{x}} \dd{\mu(x)}\\
      &= \int_{X}\E_{A} \lr{ \sum_{\sigma\in \Gamma}\inp{\pi(x \sigma) \xi_{d(x \sigma)}}{\xi_{x}}a_{\sigma}^{n}u_{\sigma}mym^{*}}  \dd{\mu(x)}\\
      &= \int_{X}\E_{A}(  \Phi(u_{n})mym^{*})  \dd{\mu(x)} = \inp{\Phi(u_{n})}{my^{*}m^{*}}
    \end{align*}
    whence
    \[
      \sup_{y\in (M)_{1}} |\inp{u_{n}\cdot( \xi \otimes \widehat{m})\cdot y}{\xi \otimes \widehat{m}}| \leq \norm{\Phi(u_{n})}_{2} \xrightarrow{n\to\infty} 0.
    \]
    We may then upgrade this result to any vectors $ v,w\in \B(\pi) $ by applying the argument in \cite[F.1.3]{BeVa08} to the set 
    \begin{align*}
        V = \{v\in \B(\pi): \lim_{n\to\infty}\sup_{y\in (M)_{1}} |\inp{u_{n}vy}{v}|=0 \text{ for any } (u_{n})_{n=1}^{\infty}&\\
      \text{ with } \norm{\E_{A}(m_{1}u_{n}m_{2})}_{2}\to 0 &\text{ for all }m_{1},m_{2}\in M\}.
    \end{align*}
   
    One shows that this set is closed under scalar multiplication, the action of $ M \odot M^{\rm{op}} $, addition, and is norm closed, whence it follows that $ V = \B(\pi) $. Lastly, one uses polarization to upgrade to pairs of vectors. \end{proof}

\begin{prop}\label{prop:bimwkcont}
For all $ \rG $-representations $\pi$ and $\rho$ such that $ \pi \prec \rho $, we have that
\[
    _M\B(\pi)_M \prec {}_M\B(\rho)_M.
\]
\end{prop}

The proof of this proposition is largely contained in the following lemma. To set notation, for a bimodule $ {}_M \H_M $ and $ \zeta\in \H $, let the functional $ \omega_{\zeta}\in (M \otimes_{\mathrm{max}} M^{\mathrm{op}})^{*} $ be given by 
\[
    \omega_{\zeta}(x \otimes y^{\mathrm{op}}) := \inp{x\cdot \zeta\cdot y}{\zeta}\quad\text{for }x,y\in M.
\]



\begin{lem}\label{lem:omegainspan}
    Let $ \pi:\rG\to \U(X\ast\H) $ and $ \rho:\rG\to\U(X\ast\K) $ be unitary representations of a measured groupoid $ \rG $ and set 
    \[
        D:=\Span\{au_{\sigma}: a\in A,\,\sigma\in [\rG]\}.
    \]
    Suppose for any $ \xi\in \int_{X}^{\oplus}\H_{x}\dd{\mu(x)} $, $ \psi\in L(\rG) $ we have
\begin{equation*}
 \left.\begin{aligned}
            &\text{there is a constant }  C_{\xi} > 0  \text{ such that for any $ \delta>0 $ and finite subsets }F\sub D,\,E\sub L(\rG),\\
            &\text{there exist sections } \eta_{1},\ldots, \eta_{l}\in \int_{X}^{\oplus}\K_{x}\dd{\mu(x)}\text{ with } \sum_{i=1}^{l}\inp{\eta_{i}}{\eta_{i}}_{A}\leq C_{\xi}\\
            &\text{and } \left| \omega_{\xi \otimes \widehat{\psi}}(x^{\prime} \otimes y) - \sum_{i=1}^{l} \omega_{\eta_{i} \otimes \widehat{\psi}}(x^{\prime}\otimes y)\right| < \delta\quad \text{for all }x^{\prime}\in F,\,y\in E.
        \end{aligned}
 \right\}
 \quad (\ast\ast) 
\end{equation*}
Then the functional $ \omega_{\xi \otimes \widehat{\psi}} $, as an element of $ (L(\rG) \otimes_{\mathrm{max}} L(\rG)^{\mathrm{op}})^{*} $, is in the set
\[
    \cls{\mathrm{Span}}^{\mathrm{wk}^*}\left\{\omega_{\zeta} : \zeta\in \B(\rho)\right\}.
\]

\end{lem}

\begin{proof}
    Without loss of generality, we assume that $ \norm{\xi}_{\infty}\leq 1 $ and $ \norm{\psi}\leq 1 $. Fix $ \eps>0 $, finite subsets $ F,E \sub L(\rG)$, and write $ F=\{x_{1},\ldots, x_{k}\} $. 
    As $ D $ is a $\mathrm{WOT} $-dense $ * $-subalgebra of $ L(\rG) $, we may apply the Kaplansky Density Theorem to find $ F^{\prime}=\{x_{1}^{\prime},\ldots, x_{k}^{\prime}\}\sub D $ such that $ \norm{x_{j}^{\prime}}\leq \norm{x_{i}} $ and 
    \[
        \norm{x_{j}^{\prime}-x_{j}}_{2} <\frac{\eps}{4 (\max\limits_{x\in F} \norm{x}^{2}+1)(\max\limits_{y\in E}\norm{y}+\norm{y}^{2}+1)(C_{\xi}+1)}=:\delta.
    \]
    Then for $ 1\leq j \leq k $ and $ y\in E $, we estimate
    \begin{align*}
        \norm{(x_{j}^{\prime}-x_{j})\cdot(\xi \otimes \widehat{\psi})\cdot y}_{2}^{2} &\leq \norm{y}^2 \inp{(x_j^{\prime}-x_j)^*(x_j^{\prime} -x_j)(\xi\otimes \widehat{\psi})}{\xi\otimes \widehat{\psi}}\\
        &= \norm{y}^2 \tau\lr{\psi^*\psi \Phi_{\xi}((x_j^{\prime}-x_j)^*(x_j^{\prime} -x_j))}\\
        &\leq\norm{y}^2 \norm{ \Phi_{\xi}((x_j^{\prime}-x_j)^*(x_j^{\prime} -x_j))}_2^2\\
        &\leq\norm{y}^2 \norm{\Phi_{\xi}}^2 \norm{ (x_j^{\prime}-x_j)^*(x_j^{\prime} -x_j)}_2^2\\
        &\leq 4 \norm{x_{j}}^{2}\norm{y}^{2} \norm{\xi}_{\infty}^{4} \norm{x_{j}^{\prime}-x_{j}}_{2} <\eps
    \end{align*}

    Apply condition $ (\ast\ast) $ to $ \delta,\,F^{\prime},\,E $ to obtain $ \eta_{1},\ldots, \eta_{l} $. Then for $ 1\leq j \leq k $ and $ y\in E $, we apply Lemma \ref{lem:phicringe}, Cauchy-Schwarz, and the Kadison-Schwarz inequality (see \cite[Proposition 3.3]{Pa02}) to compute
    \begin{align*}
        \left|\sum_{i=1}^{l} \omega_{\eta_{i}\otimes \widehat{\psi}}(x_{j} \otimes y) - \sum_{i=1}^{l} \omega_{\eta_{i}\otimes \widehat{\psi}}(x_{j}^{\prime}\otimes y) \right| &= \left|\sum_{i=1}^{l} \omega_{\eta_{i}\otimes \widehat{\psi}}((x_{j}-x_{j}^{\prime}) \otimes y) \right|\\
        &= \left|\sum_{i=1}^{l} \inp{(x_{j}-x_{j}^{\prime})\cdot(\eta_{i}\otimes \widehat{\psi y})}{\eta_{i}\otimes \widehat{\psi}} \right|\\
        &= \left|\tau\left( \psi y^{*} \psi^{*} \sum_{i=1}^{l}\Phi_{\eta_{i}}(x_{j}-x_{j}^{\prime})\right) \right|\\
        &\leq\norm{\psi y^{*} \psi^{*} }_{2} \norm{\sum_{i=1}^{l}\Phi_{\eta_{i}}(x_{j}-x_{j}^{\prime})}_{2}\\
        &\leq \norm{y} \norm{\sum_{i=1}^{l}\Phi_{\eta_{i}}(x_{j}-x_{j}^{\prime})}_{2}\\
        &\leq \norm{y} \norm{\sum_{i=1}^{l}\inp{\eta_{i}}{\eta_{i}}_{A}}_{\infty} \norm{x_{j}-x_{j}^{\prime}}_{2} \leq C_{\xi}\norm{y} \delta < \eps.
    \end{align*}
    Now finally, we estimate, for $ 1\leq j \leq k $ and $ y\in E $ 
    \begin{align*}
        &\left| \omega_{\xi \otimes \widehat{\psi}}(x_{j} \otimes y) - \sum_{i=1}^{l} \omega_{\eta_{i} \otimes \widehat{\psi}}(x_{j} \otimes y)\right|  \\
        &\leq |\omega_{\xi \otimes \widehat{\psi}}((x_{j}-x_{j}^{\prime}) \otimes y)| + |\omega_{\xi \otimes \widehat{\psi}}(x_{j}^{\prime} \otimes y) - \sum_{i=1}^{l} \omega_{\eta_{i} \otimes \widehat{\psi}}(x_{j} \otimes y)|\\
        &\leq \eps + |\omega_{\xi \otimes \widehat{\psi}}(x_{j}^{\prime} \otimes y) - \sum_{i=1}^{l} \omega_{\eta_{i} \otimes \widehat{\psi}}(x_{j}^{\prime} \otimes y)|+ |\sum_{i=1}^{l} \omega_{\eta_{i} \otimes \widehat{\psi}}(x_{j}^{\prime} \otimes y) - \sum_{i=1}^{l} \omega_{\eta_{i} \otimes \widehat{\psi}}(x_{j} \otimes y)|\\
        &\leq \eps + \delta + \eps \leq 3\eps.
    \end{align*}
    Hence we have shown that, for $ E,F\sub M $ finite subsets and $ \eps>0 $, there exists $ \omega\in \Span\{\omega_{\eta \otimes \widehat{\psi}}: \eta\in \int_{X}^{
    \oplus}\K_{x}\dd{\mu(x)},\, \norm{\eta}_\infty \leq 2\} $ such that 
    \[
        |\omega_{\xi \otimes \widehat{\psi}}(x \otimes y) - \omega(x \otimes y)| < \eps\quad \text{for all }x\in F,\, y\in E.
    \]

    By linearity, uniform boundedness, and $ \mathrm{max} $-norm continuity of the functionals involved, this result upgrades to finite subsets $ F\sub \cls{M\odot M^{\mathrm{op}}}^{\norm{\cdot}_{\mathrm{max}}} = M \otimes_{\mathrm{max}} M^{\mathrm{op}}$. Hence, we have shown that 
    \[\omega_{\xi \otimes \widehat{\psi}} \in \cls{\Span}^{\mathrm{wk}^{*}}\{\omega_{\zeta}: \zeta\in \B(\rho)\}, \]
 and the proof is finished. \end{proof}

\begin{proof}[Proof of Proposition \ref{prop:bimwkcont}]
        Without loss of generality, let $ \xi\in S_{1}(X\ast\H) $ and fix $ \delta>0 $. Suppose $ E = \{a_{1}u_{\sigma_{1}}, \ldots, a_{p}u_{\sigma_{p}}\} \sub D$ and $ F\sub L(\rG) $ is finite. Set $ \Sigma:=\{\sigma_{1},\ldots, \sigma_{p}\} $, $ Q:= \bigcup_{i=1}^{p} \sigma_{i} $. Without loss of generality, we assume that the identity element $ X = \rG^{(0)} $ is in $ \Sigma $, so $ X\sub Q $. Consider the quantity

    \[
        \eps:= \frac{\delta}{4(1 + \max\limits_{(au_{\sigma}, y)\in E\times F}\norm{\E_{A}(a u_{\sigma} \psi y \psi*)})}.
    \]

    As $ \mu_{\rG}(Q) \leq p < +\infty $, we may apply the weak containment condition with $ Q, \xi, \eps $ to obtain $ \eta^{1},\ldots, \eta^{s}\in S(X\ast\K) $ such that the set 
    \[
        Q_{\eps}:=\{g\in Q: |\inp{\pi(g) \xi_{\d(g)}}{\xi_{\r(g)}} -\sum_{i=1}^{s} \inp{\rho(g)\eta^{i}_{\d(g)}}{\eta^{i}_{\r(g)}}| > \eps\}.
    \]
    has $ \mu_{\rG}(Q_{\eps}) \leq \eps $. Then $ \mu_{\rG}(X\cap Q_{\eps}) = \mu_{\rG}(X\cap Q_{\eps}) \leq \eps $, whence the set $ S_{\eps}:= X\setminus (X\cap Q_{\eps}) $ has $ \mu(S_{\eps}) > 1-\eps $. Moreover, as $ X\sub Q $, this set has the form
    \[
        S_{\eps} = \{x\in X : |1- \sum_{i=1}^{s} \norm{\eta^{i}_{x}}| \leq \eps\}.
    \]
    For all $ x\in S_{\eps} $, we have $\sum_{i=1}^{s} \inp{\eta^i}{\eta^i}_A(x) =  \sum_{i=1}^{s} \norm{\eta^{i}_{x}}^{2} < 1+ \eps \leq 2 $, whence the sections given by $ \widetilde{\eta}_{x}^{i}:= 1_{S_{\eps}}(x) \eta^{i}_{\xi}  $ are bounded.\\

    Then, for $ au_{\sigma}\in E $ and $ y\in F $, setting $ f := \E_{A}(au_{\sigma} \psi y \psi^{*}) $ we have on one hand that
    \begin{align*}
        \bigg|\int_{\r(Q_{\eps})}& \E_{A}(au_{\sigma}\psi y \psi^{*})\cdot \left(\inp{\pi(x \sigma) \xi_{\d(x \sigma)}}{\xi_{x}} - \sum_{i=1}^{s} \inp{\rho(g) \widetilde{\eta}^{i}_{\d(x \sigma)}}{\widetilde{\eta}^{i}_{x}}\right)\dd{\mu(x)}\bigg|  \\
        &\leq \int_{\r(Q_{\eps})}|f(x)| \lr{1 + \sum_{i=1}^{s} \norm{\widetilde{\eta}^{i}_{x}}^{2}} \dd{\mu(x)} \leq 3 \eps \norm{f}_{\infty}.
    \end{align*}
    On the other hand, suppose that $ x\in X\setminus \r(Q_{\eps}) $. Then by definition $ x \sigma\not \in Q_{\eps} $, and we know that $ x \sigma\in \sigma\sub Q $. Hence, we compute

    \begin{align*}
        \bigg|\int_{X\setminus \r(Q_{\eps})}& f(x)\cdot \lr{\inp{\pi(x \sigma) \xi_{\d(x \sigma)}}{\xi_{x}} - \sum_{i=1}^{s} \inp{\rho(g) \widetilde{\eta}^{i}_{\d(x \sigma)}}{\widetilde{\eta}^{i}_{x}}}\dd{\mu(x)}\bigg|  \\
        &\leq  \mu(X\setminus \r(Q_{\eps})) \eps \norm{f}_{\infty} \leq \eps\norm{f}_{\infty}.
    \end{align*}
    Combining these estimates, we obtain that
    \begin{align*}
        |\inp{au_{\sigma}\cdot (\xi \otimes \widehat{\psi})\cdot \phi}{\xi \otimes \widehat{\psi}} - \sum_{i=1}^{s} \inp{au_{\sigma}\cdot (\widetilde{\eta}^{i} \otimes \widehat{\psi})\cdot \phi}{\widetilde{\eta}^{i} \otimes \widehat{\psi}} | < 4 \eps \norm{\E_{A}(au_{\sigma} \psi y \psi^{*})}_{\infty} \leq \delta 
    \end{align*}
    as desired. By linearity, this result extends to finite subsets $E\sub \Span\{au_\sigma: a\in A, \sigma\in [\rG]\}$, so we have shown the conditions of \ref{lem:omegainspan} apply.
    
    Consider now the set $ V $ of vectors $ v\in B(\pi) $ such that $ \omega_{v}\in \cls{\Span}^{wk^{*}}\{ \omega_{\eta}:\eta\in B(\rho)\} $. By the argument in \cite[Lemma F.1.3]{BeVa08}, $ V $ is actually closed under sums and scaling and is closed in $ B(\pi) $. Define a span-dense subset $ \mathscr{S}\sub \B(\pi) $ by 
    \[
        \mathscr{S}:=\{\xi \otimes \widehat{\psi}\in \B(\pi): \xi\in \int_{X}^{\oplus}\H_{x}\dd{\mu(x)},\, \norm{\xi}_{\infty}<+\infty,\, \psi\in L(\rG)\}.
    \]
    Applying Lemma \ref{lem:omegainspan}, we see that for $ \zeta\in \mathscr{S} $, we have that $ \omega_{\zeta}\in \cls{\Span}^{wk^{*}}\{ \omega_{\eta}:\eta\in B(\rho)\} $. Hence $ \mathscr{S} $ is contained in $ V $, so $ V = \B(\pi) $. 
\end{proof}

\section{Primeness}\label{primeness}

We now have all the necessary ingredients to prove our primeness result, so we proceed to do so. Throughout this chapter, we will set $M:=L(\rG)$, $A:=L^\infty(X,\mu)$, and $\widetilde{M}:=L(\widetilde{\rG})$ to denote the Gaussian construction applied to $\rG$ and some $1$-cocycle which will be clear in context.

\begin{lem}\label{lem:primelem}
     Let $\rG$ be a discrete measured groupoid which admits a $ 1 $-cocycle into a mixing orthogonal $ \rG $-representation $ \pi $ which is weakly contained in the regular representation. Then 
     \begin{enumerate}
         \item[(i)] $_M L^2(\widetilde{M}) \ominus L^2(M)_M$ is mixing with respect to $A$.
         \item[(ii)] $_M L^2(\widetilde{M}) \ominus L^2(M)_M  \prec {}_M L^2(M)\otimes L^2(M)_M$.
     \end{enumerate}
     If in addition the cocycle is strongly unbounded, then the Gaussian deformation $ \alpha_{t} $ with respect to this cocycle does not converge to the identity uniformly on the unit ball of $ M $.
\end{lem}

\begin{proof}
    
 Note first that $\widehat{\pi}$ is mixing and weakly contained in $\lambda_\rG$. By the construction (see Section \ref{gaussianc}) of the Gaussian bundle $X\ast \Omega$, we have a direct integral decomposition
    \[
        L^2(X\ast \Omega) \ominus L^2(X) \cong \int_X^\oplus [L^2(\Omega_x) \ominus \C] \dd{\mu(x)} = \int_X^\oplus \K_x \dd{\mu(x)}.
    \]
    Tensoring over $A$ with $L^2(\rG)$, we obtain
    \begin{equation}\label{orthocomp}
        _M L^2(\widetilde{M}) \ominus L^2(M)_M \cong [L^2(X\ast \Omega) \ominus L^2(X)]\otimes_A L^2(\rG) \cong \B(\rho).
    \end{equation}

    By Lemma \ref{fockspace} and Proposition \ref{prop:bimwkcont}, we know that $\B(\rho) \cong \B(\widehat{\pi})$ and $\B(\widehat{\pi}) \subwk \B(\lambda)$ as $M$-$M$ bimodules. Moreover, Proposition \ref{prop:bimmixing} and the identification \eqref{orthocomp} imply that $_M L^2(\widetilde{M}) \ominus L^2(M)_M$ is mixing with respect to $A$. By assumption, $\B(\widehat{\pi})\subwk \B(\lambda)$. Observe that, as $M$-$M$ bimodules, 

    \[
        \B(\lambda) = \int_X^{\oplus} \ell^2(\rG^x) \dd{\mu(x)} \otimes_A L^2(\rG) \cong L^2(M) \otimes_A L^2(M)
    \]
    Since $A$ is amenable, $L^2(M)\otimes_A L^2(M)$ is weakly contained in the coarse. Hence, we have shown items \emph{(i)} and \emph{(ii)}.

    Assume now that $ b $ is a strongly unbounded cocycle into $ \pi $. Then there exists a $ \delta>0 $ such that for all $ R>0 $, there is some full group element $ \sigma\in [\rG] $ such that $ \mu(\{\norm{b(x \sigma)}\geq R\}) > \delta $. Without loss of generality, assume $ \delta < 2 $.

    For $ \phi\in M $, we compute that
    \begin{equation}\label{popatransez}
        \norm{\alpha_{t}(\phi) - \E_{M}(\alpha_{t}(\phi))}_{2} \leq \norm{\alpha_{t}(\phi) - \phi}_{2} + \norm{\E_{M}(\phi - \alpha_{t}(\phi))}_{2} \leq 2 \norm{\alpha_{t}(\phi) - \phi}_{2}.
    \end{equation}

    Suppose, for the sake of contradiction, that $ \alpha_{t}\to {\rm id }$ uniformly on $ (M)_{1} $. Choose $ t_{0} > 0 $ such that
    \[
        \sup_{\phi\in (M)_{1}} \norm{\alpha_{t_{0}}(\phi) - \phi}_{2} < \frac{1}{2}\lr{1-\sqrt{1-\frac{\delta}{2}}} =: \gamma
    \]
    i.e. so that $ (1-2\gamma)^2 > 1-\frac{\delta}{2} $. Then, for $ \sigma\in [\rG] $, we apply \eqref{popatransez} and compute
    \begin{align*}
        \norm{\E_{M}(\alpha_{t_{0}}(u_{\sigma}))}_{2} &\geq \norm{\alpha_{t_0}(u_{\sigma})}_{2} -\norm{\alpha_{t_0}(u_{\sigma}) - \E_{M}(\alpha_{t_0}(u_{\sigma}))}_{2}\\
        & \overset{\eqref{popatransez}}{\geq} 1 - 2\norm{\alpha_{t_0}(\phi) - \phi}_{2} > 1-2\gamma,
    \end{align*}
    whence $  \norm{\E_{M}(\alpha_{t_{0}}(u_{\sigma}))}_{2}^{2} > 1-\frac{\delta}{2} $. Now choose $ R>0 $ big enough and such that $ e^{-2t_{0}^{2}R^2} < \frac{\delta}{2} $ and $ \sigma\in [\rG] $ with $ \mu(\{\norm{b(x \sigma)}\geq \sqrt{R}\}) < \delta $,
    \begin{align*}
        1-\frac{\delta}{2} \leq \norm{\E_{M}(\alpha_{t_{0}}(u_{\sigma}))}_{2}^{2} &= \norm{f_{c_{t_0},\sigma}u_{\sigma}}_{2}^{2} = \tau ( f_{c_{t_{0}},\sigma} \cls{ f_{c_{t_{0}},\sigma}}) = \int_{X} e^{-2t_{0}^2 \norm{b(x \sigma)}^{2}} \dd{\mu(x)} \\
        &\leq \int_{\{\norm{b(x \sigma)}^{2}< R\}} \dd{\mu(x)} + \int_{\{\norm{b(x \sigma)}^{2} \geq R\}}e^{2t_{0}^{2}R^{2}}\dd{\mu(x)}\\
        &\leq \mu(\{\norm{b(x \sigma)}^{2}< R\}) + e^{-2t_{0}^{2}R^{2}} \mu(\{\norm{b(x \sigma)}^{2}\geq R\})  < 1- \frac{\delta}{2},
    \end{align*}
    which is absurd.
\end{proof}

\begin{thm}\label{prime}
    Let $\rG$ be a discrete measured groupoid with no amenable direct summand which admits a strongly unbounded $1$-cocycle into a mixing orthogonal representation weakly contained in the regular representation. If $L(\rG)$ is a factor, then $L(\rG) \not \cong N\cls\otimes Q$ for any type $\mathrm{II}$ von Neumann algebras $N$ and $Q$. 
\end{thm}

\begin{proof}
   By Lemma \ref{lem:primelem} and the assumption of nowhere amenability, we may apply Popa's spectral gap argument (Lemma \ref{popaspectralgap}) and conclude that $ M $ cannot be decomposed as a tensor product of two type II von Neumann algebras.
\end{proof}

We finish the section with more concrete examples/applications of our primeness result. 

\begin{cor}\label{corcross}
    Let $G$ be a countable group which admits an unbounded $1$-cocycle into a mixing orthogonal representation weakly contained in the regular representation, and let $G\curvearrowright (X,\mu)$ be an ergodic, non-amenable probability measure preserving action. Then $L^\infty(X,\mu)\rtimes G$ is prime.
\end{cor}
\begin{proof}
    It is enough to apply Theorem \ref{prime} to the transformation groupoid $\rG=X\rtimes G$, which satisfies all the hypotheses by the observations made in Example \ref{exstrongunbound}.
\end{proof}

In particular, this corollary applies to the so-called \emph{generalized Bernoulli shifts}. The idea is to consider an action of the countable group $G$ on a countable set $I$ and a measure space $(Y,\nu)$. With these ingredients, one builds the action $\theta:G\to {\rm Aut}(Y^{I},\nu^{\otimes I})$ given by $\theta_g\big((x_i)_{i\in I}\big)=(x_{g^{-1}i})_{i\in I}$. In this setting, we say that the action of $G$ on the countable set $I$ is amenable (compare to Definition \ref{zimmer}), if it satisfies the following F{\o}lner-type condition: For every finite $F\subset G$ and $\varepsilon>0$, there exists a finite set $J\subset I$ such that
$$
|gJ\Delta J|<\varepsilon|J|, \quad\text{ for all }g\in F.
$$
For generalized Bernoulli shifts, which are an important example on their own, we have the following result.

\begin{cor}[Corollary \ref{maincor3}]\label{genbern}
    Let $G$ be a countable group which admits an unbounded $1$-cocycle into a mixing orthogonal representation weakly contained in the regular representation, and let $G\curvearrowright (X,\mu)=(Y^I,\nu^{\otimes I})$ be a generalized Bernoulli shift. Further suppose that $G\curvearrowright I$ is nonamenable. Then $L^\infty(X,\mu)\rtimes G$ is a prime factor.
\end{cor}
\begin{proof}
    Since all actions of amenable groups are amenable, $G$ has to be nonamenable and $G\curvearrowright I$ has to have infinite orbits, making $G\curvearrowright (X,\mu)$ ergodic \cite[Proposition 2.1]{KeTs08}. Furthermore, since $G\curvearrowright I$ is nonamenable, $G\curvearrowright (X,\mu)$ has to be non-amenable too. Indeed, if $G\curvearrowright (X,\mu)$ was amenable, then by \cite[pag. 752]{Ku94} and \cite[Theorem 1.2]{KeTs08}, the left regular representation of $G$ would have almost invariant vectors, making $G$ amenable. Now the claimed result follows from the previous corollary.
\end{proof}

\section{Fullness}\label{fullness}

As mentioned before, our purpose for this section is to provide conditions that imply that the groupoid von Neumann algebra $L(\rG)$ is a full factor. Let us recall the definition of fullness, which was introduced by Connes \cite{Co74} as the negation of Murray-von Neumann's property Gamma \cite{MuvN43}.

\begin{defn}
    A factor with separable predual $M$ will be called \emph{full} if for any uniformly bounded sequence $\{\psi_n\}_{n=1}^\infty \subset M$ with $\norm{\psi_n\phi-\phi \psi_n}_2\to 0$ for any $\phi\in M$, must satisfy $\norm{\psi_n-\tau(\psi_n)1}_2\to 0$.
\end{defn}

In particular, it will be a standing assumption of the section that our groupoid $\rG$ is such that $L(\rG)$ is a factor. Since $L(\rG)$ admits a faithful tracial state, it is forced to be a ${\rm II}_{1} $-factor. The idea now is to provide conditions that ensure that $L(\rG)$ is a full factor. Let us note that a necessary condition for fullness is that of `strong ergodicity', which is defined as follows.

\begin{defn}
    A discrete measured groupoid $\rG$ is called \emph{strongly ergodic} if for any sequence of measurable subsets $Y_n\subset X$ such that 
    \[
    \lim_{n\to\infty} \mu(\si Y_n\si^{-1}\Delta Y_n)=0 \text{ for every } \si\in[\rG],
    \] we have $\lim_{n\to\infty} \mu(Y_n)\big(1-\mu(Y_n)\big)=0$.
\end{defn}

Before diving into our fullness result, let us state two lemmas that will be useful for its proof. The first one is a very simple statement coming from ergodic theory.

\begin{lem}\label{ergodicth}
    Let $(X,\mu)$ be a probability space and $f\in {\rm Aut}(X,\mu)$, with $f^2={\rm id}_X$ such that $F={\rm Fix}_f^c=\{x\in X :  f(x)\not=x\}$ has positive measure. Then there is a subset $E\subset F$ of positive measure such that $F\setminus\big(E\sqcup f(E)\big)$ is null.
\end{lem}
\begin{proof}
Consider the equivalence relation $\mathcal R\subset F\times F$ given by the orbits of $f$ and let $\si\in[\mathcal R]$ be any non-trivial full group element. Projecting $\si\setminus {\rm Diag}(F)$ onto the first coordinate gives a set of positive measure $E\subset X$ that has to satisfy $\mu\big(f(E)\cap E\big)=0$.

Now, by Zorn's lemma, one can find a maximal family of subsets $\{E_i\}_{i\in I}$ with positive measure and such that $\mu\big(f(E_i)\cap E_j\big)=0$ for all $i,j\in I$ and $E_i\cap E_j=\emptyset$, for $i\not=j$. Such a family must necessarily be countable, and their union $E=\bigcup_{i\in I}E_i$ is the required subset.
\end{proof}

In order to prove the next lemma, we need to recall Sorin Popa's incredibly powerful \textit{intertwining-by-bimodules} technique \cite[Theorem 2.1]{Po03}.
\begin{thm}[Popa's Intertwining by Bimodules]
    Let $ (M,\tau) $ be a tracial von Neumann algebra and $ P\sub pMp $, $ Q\sub qMq $ von Neumann subalgebras. Let $ \U\sub \U(P) $ be a subgroup which generates $ P $ as a von Neumann algebra. Then the following are equivalent:
    \begin{enumerate}
        \item[(i)] There are projections $ p_{0}\in P $, $ q_{0}\in Q $, a $ * $-homomorphism $ \theta:p_{0}Pp_{0}\to q_{0}Qq_{0} $, and a nonzero partial isometry $ v\in q_{0}Mp_{0} $ such that $ \theta(x)v = vx $ for all $ x\in p_{0}Pp_{0} $.
        \item[(ii)] There does not exist a sequence $ (u_{n})_{n=1}^{\infty} $ in $ \U $ such that $ \norm{\E_{Q}(xu_{n}y)}_{2}\to0 $ for all $ x,y\in M $.
    \end{enumerate}
    If either of above equivalent conditions are satisfied, we write $ P\prec_{M}Q $.
\end{thm}

The following lemma is well-known and largely due to Ioana, Popa and Vaes \cite[Lemma 2.3]{IoPoVa10}, the difference here being that we ensure the centrality of the projections involved. 

\begin{lem}\label{standard}
   Let $ P\sub pMp $ be a von Neumann subalgebra, $ _M \H_{P} $ an $ M $-$ P $ bimodule, and $ \kappa>0 $. Suppose that there is a sequence $ (\xi_{n})_{n=1}^{\infty} $ in $ \H $ such that
   \begin{enumerate}
       \item[(i)] $ \norm{[a,\xi_{n}]} \xrightarrow{n\to\infty}0 $ for all $ a\in P $,
       \item[(ii)] $ \norm{x \xi_{n}}\leq \kappa \norm{x}_{2} $ for all $ x\in M $,
       \item[(iii)] $ \limsup_{n\to\infty}\norm{p \xi_{n}}>0 $.
   \end{enumerate}
   Then there is some nonzero projection $ z\in \mathcal{Z}(P^{\prime}\cap pMp) $ such that $_M L^{2}(M) _{Pz} $ is weakly contained in $ _M \H_{Pz} $.
\end{lem}

\begin{proof}
    Without loss of generality, replace $ \xi_{n} $ with $ p \xi_{n} $ so that we may assume $ \xi_{n}=p \xi_{n} $.

    Consider the functional $ \phi_{n}:M\to \C $ given by $ \phi_{n}(x)=\inp{x \xi_{n}}{\xi_{n}} $. Then note that $ 0\leq \phi_{n}\leq \kappa^{2} \tau $, so by Radon-Nikodym there is some $ T_{n}\in pM_{+}p $ such that 
    \[
        \tau(xT_{n}) = \phi_{n}(x) = \inp{x \xi_{n}}{\xi_{n}} \quad \text{for all }x\in M.
    \]
    After passing to subsequences, we may assume there is some $ T\in pM_{+}p $ such that $ T_{n}\to T $ weakly and $ \tau(T) > 0 $. 

    Suppose $ x,y,a\in M $ and set $ m=xy^{*} $. Then we compute 
    \begin{align*}
        |\inp{[a,T_{n}]\widehat{x}}{\widehat{y}}| &= |\tau(y^{*}aT_{n}x)-\tau(y^{*}T_{n}ax)| \\
        &= |\tau(maT_{n}) - \tau(amT_{n})|\\
        &= |\inp{[m,a] \xi_{n}}{\xi_{n}}| \\
        &\leq |\inp{m[a,\xi_{n}]}{\xi_{n}}| + |\inp{m \xi_{n}a - am \xi_{n}}{\xi_{n}}|\\
        &\leq \norm{m}\norm{[a,\xi_{n}]}\norm{\xi_{n}} +|\inp{m \xi_{n}}{[\xi_{n},a^{*}]}|\\
        &\leq  \norm{m}\norm{[a,\xi_{n}]}\norm{\xi_{n}} + \norm{m}\norm{[\xi_{n},a^{*}]}\norm{\xi_{n}}
    \end{align*} 
    Letting $ J $ denote the modular conjugation operator, we observe that 
    \[
        \norm{[\xi_{n},a^{*}]} = \norm{J[\xi_{n},a^{*}]} = \norm{J(JaJ \xi_{n})-\xi_{n}a} =\norm{[a,\xi_{n}]},
    \]
    whence, we finally obtain 
    \[
        |\inp{[a,T_{n}] \widehat{x}}{\widehat{y}}| \leq 2\norm{m}\norm{\xi_{n}}\norm{[a,\xi_{n}]}.
    \]
    By condition (3) and uniqueness of limits, it follows that $ T\in P^{\prime}\cap pMp $.

    Now choose $ \delta>0 $ small enough such that $ 1_{(\delta, \kappa] }(T^{1/2}) $ is a nonzero projection, and define $ f(t)\coloneqq t^{-1}1_{(\delta, \kappa]}(t) $ for $ t\in \sigma(T^{1/2}) $. Then set $ S\coloneqq f(T^{1/2}) $ and $ q\coloneqq TS^{2} $. 


    By \cite[Part III, Chapter 8.2]{Di81}, there is some nonzero subprojection $ \widetilde{q} $ of $q$ which is fundamental, i.e there exist nonzero, pairwise disjoint projections $ q_{1}, \ldots, q_{n}\in P^{\prime}\cap pMp $ and a central projection $ z\in \mathcal{Z}(P^{\prime}\cap pMp) $ such that
    \begin{itemize}
        \item $ \widetilde{q}=q_{1}\sim q_{2}\sim \cdots\sim q_{n} $
        \item $ \sum_{j=1}^{m} q_{j} = z $.
    \end{itemize}
    Then, as $ \E_{\mathcal{Z}(P^{\prime}\cap pMp)}(\cdot) $ is a center-valued trace, we have that
    \begin{align*}
        z = \E_{\mathcal{Z}(P^{\prime}\cap pMp)}(z)= \E_{\mathcal{Z}(P^{\prime}\cap pMp)}\lr{\sum_{i=1}^{m} q_{i} } = m\cdot\E_{\mathcal{Z}(P^{\prime}\cap pMp)}(\widetilde{q}) 
    \end{align*}
    Choose partial isometries $ v_{1},\ldots, v_{m}\in P^{\prime}\cap pMp $ such that $ v_{i}^{*}v_{i}=\widetilde{q} $ and $ v_{i}v_{i}^{*} = q_{i} $ for all $ 1\leq i \leq m $. As we are now working under $ \widetilde{q} $, we adujust the operator $ S $ via $ \widetilde{S}\coloneqq S \widetilde{q} $, so $ \widetilde{S}^{2}T = \widetilde{q} $.
    For $ x,a \in M $ and $ y\in Pz $, observe that
    \begin{align*}
        \inp{x \widehat{a} y}{\widehat{a}} &= \tau(a^{*}xay) = \tau(a^{*}xayz) =\sum_{i=1}^{m} \tau(a^{*}xayv_{i}v_{i}^{*})\\
        &=\sum_{i=1}^{m} \tau(a^{*}xav_{i}yv_{i}^{*})=\sum_{i=1}^{m} \tau(a^{*}xayv_{i}v_{i}^{*}v_{i}v_{i}^{*})=\sum_{i=1}^{m} \tau(a^{*}xayv_{i}\widetilde{S}T \widetilde{S}v_{i}^{*})\\
        &= \lim_{n\to\infty} \sum_{i=1}^{m}\tau(\widetilde{S}v_{i}^{*}a^{*}xayv_{i} \widetilde{S}T_{n})\\
        &=\lim_{n\to\infty}\sum_{i=1}^{m} \inp{\widetilde{S}v_{i}^{*}a^{*}xay v_{i}\widetilde{S} \xi_{n}}{\xi_{n}} =\lim_{n\to\infty}\sum_{i=1}^{m} \inp{xa v_{i}\widetilde{S} \xi_{n} y}{av_{i}\widetilde{S}\xi_{n}}. 
    \end{align*}
    To show weak containment, suppose $ F_{1}\sub M $, $ F_{2}\sub Pz $ are finite subsets, $ \eps>0 $, and $ a\in M $. Set $ \eta_{n}^{i}\coloneqq av_{i}\widetilde{S} \xi_{n}\in \H $ for all $ n\in \N $ and $ 1\leq i \leq m $. Since these subsets are finite, we may choose $ k\in \N $ such that $  $
    \[
        \left| \inp{x \widehat{a} y}{\widehat{a}} - \sum_{i=1}^{m}\inp{x \eta_{k}^{i}y}{\eta_{k}^i}\right| < \eps\quad \text{for all }x\in F_{1},\,y\in F_{2}.
    \]
    By density of $ M $ inside $ L^{2}(M) $, this shows that we can in fact approximate any vector state in $ L^{2}(M) $ by vector states from $ \H $, whence we conclude that $\prescript{}{M}{L^{2}(M)}_{Pq}\prec\prescript{}{M}{\H}_{Pq}$.
\end{proof}
The preceding lemma presents a certain relative amenability condition for direct summands. In the opposite case of strong nonamenability, one can use the contrapositive of the previous lemma to find unitaries in $ P $ for which commutation with these unitaries detects closeness to $ M $. 

We will now put everything together to prove the main result of this section.

\begin{thm}\label{full}
Let $\rG$ be a strongly ergodic discrete measured groupoid which is nonamenable and admits a strongly unbounded $1$-cocycle into a mixing orthogonal representation weakly contained in the regular representation. If $L(\rG)$ is a factor, then it is a full factor. 
\end{thm}
\begin{proof}
Let $M,\widetilde{M}$ and $\{\alpha_t\}_{t\in\mathbb R} \subset \Aut(\tilde M)$ as constructed in Section \ref{gaussianc} and suppose that $M$ has property Gamma. Then there is a sequence $\{u_n\} \in \U(M)$ with $\tau(u_n) = 0$ for all $n$ and $$\lim_{n \to \infty}\norm{u_nx-xu_n}_2 =0,$$ for each $x \in M$. 
Then for any $u \in \mathcal N_M(A)$ we have 
\begin{align*}
\norm{u\E_A(u_n)u^*-\E_A(u_n)}_2 =\norm{\E_A(uu_nu^*)-\E_A(u_n)}_2\leq\norm{uu_nu^* - u_n}_2,
\end{align*}
which goes to $0$ as $n\to\infty$. Since the sequence $\E_A(u_n)$ is bounded in norm and $M=\mathcal{N}_M(A)''$, it follows that $\norm{x \E_A(u_n) - \E_A(u_n)x}_2 \to 0$ for each $x \in M$. Since $\rG$ is strongly ergodic, it follows that $\norm{\E_A(u_n)}_2=\norm{\E_A(u_n)-\tau(\E_A(u_n))1}_2 \to 0$ as $n \to \infty$.

Consider the quotient map 
$$
q:\rG\to\mathcal{R}_\rG,\quad\text{given by}\quad q(g)=\big(\r(g),\d(g)\big).
$$
Noting that $q$ is countable-to-one, by Lusin-Novikov \cite[Theorem 18.10]{Ke95} there is a Borel right inverse to $q$, we denote it by $f:\mathcal{R}_\rG\to\rG$. By replacing $f$ with $f'(x,y):=f(x,y)f(y,y)^{-1}$ if necessary, we can assume that $f(x,x)=x$, for all $x\in X$. Note that $\si=f(\widetilde\si)\in [\rG]$ as soon as $\widetilde\si\in [\mathcal{R}_\rG]$ and therefore we can write every element in $[\rG]$ as a multiplication $\si_1\si_2$, with $\si_1\in [{\rm Iso}(\rG)]$ and $\si_2=f(\widetilde{\si_2})\in f\big([\mathcal{R}_\rG]\big)$.

Hence the full group $[\rG]$ fits into a short exact sequence (of groups) in the following manner 
$$
1\to[{\rm Iso}(\rG)]\to [\rG]\xrightarrow{q} [\mathcal{R}_\rG]\to 1,
$$ 
where the map $q:[\rG]\to[\mathcal{R}_\rG]$ is induced by the original quotient map.

Now fix any $\widetilde{\si}\in [\mathcal{R}_\rG]$ with ${\widetilde{\si}}^2=X$ and let $\si=f(\widetilde\si)$. Note that $z_\si = \E_A(u_\si)$ is a projection given by $z_\si=z_{\si^{-1}}=1_Y$, where $Y={\{x\in X:{\si}x {\si}^{-1}=x\}}={\{x\in X:\widetilde{\si}x \widetilde{\si}^{-1}=x\}}={\rm Fix}_g$, where $g(x)={\si}x {\si}^{-1}$, for all $x\in X$. We now observe that 
$$
u_\si^*z_\si = \E_A(u_\si^*z_\si) = \E_A(u_\si^*)z_\si = z_\si.
$$
Hence 
\begin{align}\label{ineq1}
\norm{\E_A(u_nu_\si^*)z_\si}_2=\norm{\E_A(u_nu_\si^*z_\si)}_2=\norm{\E_A(u_n)z_\si}_2\leq \norm{\E_A(u_n)}_2 \xrightarrow{n \to \infty} 0
\end{align}
Moreover, since $1-z_\si=1_{Y^c}$ and $Y={\rm Fix}_g$, lemma \ref{ergodicth} allows us to find a projection $z\in A$ such that $1-z_\si=z + u_\si zu_\si^*$ and $z \perp u_\si zu_\si^*$. It now follows that
\begin{align}\label{ineq2}
\norm{\E_A(u_nu_\si^*)(1-z_\si)}^2_2&=\norm{\E_A(u_nu_\si^*)(z + u_\si zu_\si^*)}_2^2 \notag\\
&=\norm{\E_A(u_nu_\si^*)z}_2^2 + \norm{\E_A(u_nu_\si^*)u_\si zu_\si^*}_2^2  \notag\\
&= \norm{\E_A(u_nu_\si^*)(z-u_\si zu_\si^*)}_2^2
= \norm{z\E_A(u_nu_\si^*)-\E_A(u_nu_\si^*u_\si zu_\si^*)}_2^2 \notag\\
&= \norm{\E_A(zu_nu_\si^*-u_nzu_\si^*)}_2^2
\leq\norm{zu_n-u_nz}_2^2 \xrightarrow{n \to \infty} 0. 
\end{align} 

Combining \eqref{ineq1} and \eqref{ineq2} we see that $\norm{\E_A(u_nu_\si^*)}_2 \to 0$ as $n \to \infty$ for each $\si=f(\widetilde\si)$, where $\widetilde{\si}\in [\mathcal{R}_\rG]$ with ${\widetilde{\si}}^2=X$. By Feldman and Moore \cite{FM1:77}, the involutions $\widetilde{\si}$ from before cover $\mathcal{R}_\rG$ and so $\vN{au_{\si_1}u_{f(\widetilde{\si_2})} : a \in A, \si_1\in [{\rm Iso}(\rG)],\si_2 \in [\mathcal{R}_\rG], \widetilde{\si_2}^2 = X}$ contains all the unitaries $u_\si$ coming from $\si \in f\big([\mathcal{R}_\rG]\big)$. In consequence, it has to coincide with $L(\rG)$ and we obtain that $\norm{\E_A(xu_ny)}_2 \to 0$ as $n \to \infty$ for any pair $x, y \in M$.

Recall that, in the setting of Lemma \ref{transv}, we defined $\delta_t(x) = \alpha_t(x) - \E_M(\alpha_t(x))$. By Lemma \ref{lem:primelem}, we know that ${}_ML^2(\tilde M) \ominus L^2(M)_M$ is mixing relative to $A$, so this implies that 
\[
\<u_n\delta_t(x), \delta_t(x)u_n\> \to 0, \text{ as } n \to \infty,
\] for every $x \in M$. From Lemma \ref{lem:primelem}, we obtain that the convergence $\alpha_t \to {\rm id}$ is not uniform on $(M)_1$, and hence by part (i) of Lemma \ref{transv} there is an $\varepsilon > 0$ and sequences $\{x_k\} \subset (M)_1$, $\{t_k\} \subset \mathbb R$, $t_k \to 0$, such that $\norm{\delta_{t_k}(x_k)}_2 \geq \varepsilon$ for all $k$. Then, using Popa's transversality inequality again, we get that for any $k$,
\begin{align*}
\varepsilon^2 &\leq \norm{\delta_{t_k}(x_k)}_2^2 
= \liminf_{n \to \infty} \left[\frac{1}{2}\norm{[u_n, \delta_{t_k}(x_k)]}_2^2 + \text{Re}\<u_n\delta_{t_k}(x_k), \delta_{t_k}(x_k)u_n\> \right] \\
& \le \frac{1}{2}\liminf_{n \to \infty} \left[2\norm{\alpha_{t_k}(u_n) - u_n}_2 + \norm{[u_n, x_k]}_2\right]^2
\le 8\liminf_{n \to \infty} \norm{\delta_{t_k/2}(u_n)}_2^2
\end{align*}
Thus setting $s_k = \frac{t_k}{2}$ for each $k$ we can find $n_k \geq k$ such that $\norm{\delta_{s_k}(u_{n_k})}_2 \geq \frac{\varepsilon}{4}$. Finally, we see that, for any $x \in M$,
\[
\norm{[\delta_{s_k}(u_{n_k}), x]}_2 \leq 2\norm{\alpha_{s_k}(x) - x}_2 + \norm{[u_{n_k}, x]}_2 \to 0 \text{ as } k \to \infty.
\]
Since we also have $\norm{x\delta_{t_k/2}(u_{n_k})}_2 \leq 2\norm{x}_2$ for all $k$, we apply Lemma \ref{standard} to find that ${}_ML^2(M)_M \prec {}_ML^2(\tilde M) \ominus L^2(M)_M$. On the other hand, Lemma \ref{lem:primelem} shows that ${}_ML^2(\tilde M) \ominus L^2(M)_M \prec {}_ML^2(M) \otimes L^2(M)_M$. This implies that $M$ is amenable, which is a contradiction. 
\end{proof}

The following corollary is now immediate from the Theorem and previous considerations.
\begin{cor}
    Let $G$ be a countable group which admits an unbounded $1$-cocycle into a mixing orthogonal representation weakly contained in the regular representation and let $G\curvearrowright (X,\mu)$ a strongly ergodic, nonamenable probability measure preserving action such that $L^\infty(X,\mu)\rtimes G$ is a factor (immediate if $G$ is i.c.c. or if the action is free). Then $L^\infty(X,\mu)\rtimes G$ is a full factor.
\end{cor}

\section{Unique prime factorization}\label{UPF}

Let us recall the following Theorem due to Hoff \cite[Theorem 6.4]{Ho16}, which he used to obtain a result about product rigidity for equivalence relations. It will do the heavy lifting for our result on unique prime factorizations.

\begin{thm} \label{uniquefactor}
Let $M_1, \dots, M_k$ be $\rm{II}_1$ full factors, each with an $s$-malleable deformation $\{\alpha^i_t\}_{t\in\mathbb R} \subset {\rm Aut} (\tilde M_i)$ for some tracial von Neumann algebras $\tilde M_i \supset M_i$. Suppose that for each $i$, the $M_i$-$M_i$ bimodule $_{M_i}L^2(\tilde M_i) \ominus L^2(M_i)_{M_i}$ is weakly contained in the coarse $M_i$-$M_i$ bimodule and mixing relative to some abelian subalgebra $A_i \subset M_i$.  Assume that the convergence $\alpha^i_t \to {\rm id}$ is not uniform in $\norm{\cdot}_2$ on $(M_i)_1$ for any $i$. Then $M_i$ is prime for each $i$, and \begin{enumerate}
\item[(i)] If $M = M_1 \cls{\otimes} M_2 \cls{\otimes} \dots \cls{\otimes} M_k = N \cls{\otimes} Q$ for tracial factors $N, Q$, there must be a partition $I_N \cup I_Q = \{1, \dots, k\}$ and $t > 0$ such that  
$N^t = \bigotimes_{i \in I_N} M_i$ and $Q^{1/t} = \bigotimes_{i \in I_Q} M_i$ modulo unitary conjugacy in $M$.
\item[(ii)] If $M = M_1 \cls{\otimes} M_2 \cls{\otimes} \dots \cls{\otimes} M_k = P_1 \cls{\otimes} P_2 \cls{\otimes} \cdots \cls{\otimes} P_m$ for $\rm{II}_1$ factors $P_1, \dots, P_m$ and $m \geq k$, then $m = k$, each $P_i$ is prime, and there are $t_1, \dots, t_k > 0$ with $t_1t_2\cdots t_k = 1$ such that after reordering indices and conjugating by a unitary in $M$ we have $M_i = P_i^{t_i}$ for all $i$.
\item[(iii)] In {(ii)}, the assumption $m \geq k$ can be omitted if each $P_i$ is assumed to be prime.
\end{enumerate}
\end{thm} 

In fact, the above-mentioned theorem can be applied directly in our setting. Its application in conjunction with Theorem \ref{prime} and Theorem \ref{full} yields the following result.

\begin{cor}[Theorem \ref{mainth0}]\label{mainth1}
For $i \in \{1, 2, \ldots, k\}$, let $\rG_i$ be a nonamenable strongly ergodic discrete measured groupoid which admits a strongly unbounded $1$-cocycle into a mixing orthogonal representation weakly contained in the regular representation. Further assume that $L(\rG_i)$ is a factor. Then $M=L(\rG_1) \cls{\otimes} L(\rG_2) \cls{\otimes} \dots \cls{\otimes} L(\rG_k) $ satisfies the following.
\begin{enumerate}
    \item[(i)] If $M = N \cls{\otimes} Q$ for tracial factors $N, Q$, there must be a partition $I_N \cup I_Q = \{1, \dots, k\}$ and $t > 0$ such that $N^t = \bigotimes_{i \in I_N} L(\rG_i)$ and $Q^{1/t} = \bigotimes_{i \in I_Q} L(\rG_i)$ modulo unitary conjugacy in $M$.
    \item[(ii)] If $M = P_1 \cls{\otimes} P_2 \cls{\otimes} \cdots \cls{\otimes} P_m$ for $\rm{II}_1$ factors $P_1, \dots, P_m$ and $m \geq k$, then $m = k$, each $P_i$ is prime, and there are $t_1, \dots, t_k > 0$ with $t_1t_2\cdots t_k = 1$ such that after reordering indices and conjugating by a unitary in $M$ we have $L(\rG_i) = P_i^{t_i}$ for all $i$. 
    \item[(iii)] In (ii), the assumption $m \geq k$ can be omitted if each $P_i$ is assumed to be prime.
\end{enumerate}\end{cor}

As with Theorems \ref{prime}, \ref{full}, this theorem can be applied to the case of a transformation groupoid. We omit writing down this result.

\section{Maximal rigid subalgebras}\label{maxrigsec}

In order to finish the article, let us describe the maximal rigid subalgebras of $L(\rG)$. We will achieve that in the following two propositions.

\begin{prop} \label{maxrig1}
  Let $ \rG $ be an ergodic discrete measured groupoid and $ \pi $ an orthogonal representation of $ \rG $ on a real Hilbert bundle $ X\ast \H $. Let $ b\in Z^{1}(\rG,\pi) $ be a $1$-cocycle and set 
  $$
    \mathcal{S} = \{g\in \rG : b(g) = 0\}.
  $$
  Then $ \mathcal{S} $ is a wide (i.e $\mathcal{S}^{(0)} = \rG^{(0)}$) discrete measured subgroupoid of $ \rG $. Moreover, if $ L(\mathcal{S}) $ is diffuse and $ \pi\vert_{\mathcal{S}} $ is weak mixing, then $ L(\mathcal{S}) $ is a maximal rigid subalgebra for $ \alpha_{b} $.
\end{prop}

\begin{proof}
    We will follow the proof of \cite[Proposition 4.3]{dSHH21}. $\mathcal S$ is easily seen to be a wide subgroupoid; every unit $x\in X$ satisfies $x^2=x$ and hence the $1$-cocycle identity implies $x\in \S$. Showing closedness for multiplication and inverses is even easier. We also note that $\alpha_{b,t}\vert_{L(\S)}={\rm id}_{L(\S)}$. 
    
    We now want to show that $L(\S)'\cap \widetilde{M}\subset M$. Note that, for $\tau\in[[\rG]]$ and $\psi=\sum_{\si\in [[\widetilde\rG]]} a_\si u_\si\in L(\widetilde\rG)$, one has
    \[
        u_\tau \psi u_\tau^*=u_\tau \big(\sum_{\si\in [[\rG]]} a_\si u_\si\big)  u_\tau^* = \sum_{\si\in [[\rG]]} \theta_\tau(a_\si) u_{\tau\si\tau^{-1}}=\kappa\otimes\lambda_c(\tau)\psi,
    \]
    where $\kappa:[[\rG]]\to L^2(X\ast\Omega)$ is the Koopman representation (cf. \cite[Section 3.8]{GaLu17}) and $\lambda_c:[[\rG]]\to L^2(\rG)$ is the conjugation representation, that is, $\lambda_c(\si)\phi(x)=\phi(\si^{-1}x\si)$ for all $\phi\in L^2(\rG)$ and $\si\in [[\rG]]$. It is therefore enough to argue that the representation $(\kappa\otimes\lambda_c)|_{[\S]}$ on $L^2(\widetilde M)\ominus L^2(M)$ has no nonzero invariant vectors. But this is the case as $(\kappa\otimes\lambda_c)|_{[\S]}=\kappa|_{[\S]}\otimes\lambda_c|_{[\S]}$ and $\kappa|_{[\S]}$ is weakly mixing because of Remark \ref{Gaussmix} and the assumption on $\pi|_\S$.

    Finally, we let $P$ be the rigid envelope of $L(\S)$ under the deformation $\alpha_{b,t}$. By \cite[Theorem 3.5]{dSHH21}, we have that 
    \[
    \norm{(\alpha_{b,t}-{\rm id})|_{P}}_{\infty,2}\leq 24\norm{(\alpha_{b,t}-{\rm id})|_{L(\S)}}_{\infty,2}=0
    \] 
    for all small enough $t>0$. Since $t\mapsto \alpha_{c,t}$ is a homomorphism, we have that $\alpha_{c,t}\big|_{P}$ is the identity for all $t$. But then it follows that $P\leq L(\S)$, and therefore $P=L(\S)$.
\end{proof}

\begin{thm} \label{maxrig2}
  Let $ \rG $ be an ergodic discrete measured groupoid and $ \pi $ an orthogonal representation of $ \rG $ on a real Hilbert bundle $ X\ast \H $. Let $ b\in Z^{1}(\rG,\pi) $ be a $1$-cocycle and suppose $\mathcal{S}\leq \rG$ is a discrete measured subgroupoid with $L(\mathcal{S})$ diffuse and such that $\pi\vert_\mathcal{S}$ is weakly mixing and $b\vert_\mathcal{S}$ is bounded. Let $P$ be the rigid envelope of $L(\mathcal{S})$. Then $P=L(\mathcal{S}')$, where $\mathcal{S}\leq \mathcal{S}'\leq \rG$ and $\mathcal{S}'$ is a maximal subgroupoid satisfying that $b\vert_{\mathcal{S}'}$ is bounded.
\end{thm}

\begin{proof}
    Since $\mathcal S$ is ergodic, Lemma \ref{bound} gives a measurable section $\xi:X\to X\ast \H$, such that $b(g)=\xi_{\r(g)}-\pi(g)\xi_{\d(g)}$, for all $g\in \mathcal S$. Now define $\tilde b(g)\coloneqq b(g)-\xi_{\r(g)}+\pi(g)\xi_{\d(g)}$ and its associated subgroupoid
    $$
    \mathcal S'={\rm Ker}\,\tilde b=\{g\in\rG: \tilde b(g)=0\}.
    $$
    It follows that $\mathcal S'$ contains $\mathcal S$ and is maximal under the condition that $b\vert_{\mathcal{S}'}$ is bounded: if $\mathcal S'\leq \mathcal S''$ and $b(g)=\xi_{\r(g)}'-\pi(g)\xi_{\d(g)}'$ for $g\in \mathcal S''$, then $\xi'-\xi$ is invariant under $\pi\vert_\mathcal{S}$ and since $\pi\vert_\mathcal{S}$ is weakly mixing, $\xi'=\xi$, from where $\mathcal S'=\mathcal S''$.
    
    Now note that $\tilde b-b$ is a coboundary, so by Lemma \ref{bound}, there is $E\subset X$ of measure 1 such that
    $$
    \sup_{g\in \rG_E^x} \norm{\tilde b(g)-b(g)}<\infty.
    $$ Now let $\alpha_{\tilde b},\alpha_{b}$ be the associated $s$-malleable deformations inside $\widetilde M$. We observe \begin{align*}
        \norm{(\alpha_{\tilde b,t}-\alpha_{\tilde b,t})\big{|}_M}_{L^2(M)\to L^2(\widetilde M)}^2&\leq \sup_{\sigma\in [\rG]} 2-\tau(\overline{f_{\tilde{c}_t,\sigma}}f_{{c}_t,\sigma}+\overline{f_{{c}_t,\sigma}}f_{\tilde{c}_t,\sigma}) \\ 
        &\leq 2\sup_{\sigma\in [\rG]} \lr{1-\int_{X} e^{-t^2\norm{\tilde b(x\sigma)-b(x\sigma)}^2}\dd{\mu(x)}} \\ 
        &\leq 2-2\int_{X} e^{-t^2\sup_{g\in \rG_E^x} \norm{\tilde b(g)-b(g)}^2}\dd{\mu(x)}
    \end{align*} thus $$
    \lim_{t\to 0} \norm{(\alpha_{\tilde b,t}-\alpha_{\tilde b,t})\big{|}_M}_{\infty,2}^2=0.
    $$ So a diffuse subalgebra $Q\leq M$ is $\alpha_{b}$-rigid if and only if it is $\alpha_{\tilde b}$-rigid and by Proposition \ref{maxrig1}, $L(\mathcal S')$ is maximal rigid for $\alpha_b$ and equals $P$ (Theorem \ref{rigenvelope}).
\end{proof}

In particular, when the groupoid is a transformation groupoid $\rG=X\rtimes_\theta G$ and the $1$-cocycle comes from the acting group $G$ (that is, the setting of Example \ref{exstrongunbound}), we note that $\S'$ is of the form $\S'=X\rtimes_\theta H$, for some subgroup $H\leq G$ and therefore the rigid envelope of $L(\S)$ is a crossed product $L^\infty(X,\mu)\rtimes H$ (even when $\S$ is not assumed to be a transformation subgroupoid).

Before stating the last theorem of the paper, let us recall the definition of property (T), as taken from \cite{Po06}.

\begin{defn}
    A tracial von Neumann algebra $(M,\tau)$ is said to have \emph{property (T)} if for every $\varepsilon>0$, there exists a finite subset $F\subset M$ and a $\delta>0$ such that for every bimodule $_M\H_M$ and any tracial vector $\xi\in\H$ satisfying $\max_{x\in F}\norm{x\xi-\xi x}\leq \delta $, there exists a central vector $\eta\in\H$ such that $\norm{\eta-\xi}\leq\varepsilon.$
\end{defn}

\begin{ex}
    A group von Neumann algebra has property (T) exactly when the group has Kazhdan's property (T). In particular $L(\rm{SL}_k(\Z))$ has property (T) when $k\geq 3.$
\end{ex}

\begin{thm}\label{bens}
    Let $ \rG $ be an ergodic discrete measured groupoid and $ \pi $ a weak mixing, orthogonal representation of $ \rG $ on a real Hilbert bundle $ X\ast \H $. Suppose that $\rG$ admits an unbounded 1-cocycle into $\pi$ and that $L(\mathcal{G})$ is diffuse. Then, for any pair $P,Q\subset L(\rG)$ of subalgebras with property (T) such that $P\cap Q$ is diffuse, we have that $P\vee Q\not=L(\rG)$.
\end{thm}

\begin{proof}
    Since $P$ and $Q$ have property (T), they are rigid for any deformation \cite[Proposition 4.1]{Po06}. It follows from \cite[Theorem 1.3]{dSHH21} that $P\vee Q$ is rigid, and therefore it is contained in a maximal rigid subalgebra \cite[Theorem 1.2]{dSHH21}. But Theorem \ref{maxrig2} prevents this algebra from being the entirety of $L(\rG)$.
\end{proof}

\section*{Acknowledgments}

Felipe Flores has been supported by the NSF grants DMS-2000105 and DMS-2144739. James Harbour has been supported by the NSF grant DMS-2000105, the Ingrassia Family Research Grant, and the Harrison Research Award. Both authors would like to express their gratitude to Ben Hayes for his guidance and the uncountable interesting discussions that we have had. We would also like to thank Soham Chakraborty for valuable insight on measured groupoids.

\printbibliography

\bigskip
\bigskip
ADDRESS

\smallskip
Felipe Flores: Department of Mathematics, University of Virginia, 114 Kerchof Hall. 141 Cabell Dr,
Charlottesville, Virginia, United States. E-mail: hmy3tf@virginia.edu

\bigskip
James Harbour: Department of Mathematics, University of Virginia, Kerchof Hall. 141 Cabell Dr,
Charlottesville, Virginia, United States. E-mail: james.h.harbour@gmail.com

\end{document}